\documentclass[12pt]{article}

\pdfoutput=1

\usepackage{amsmath, amssymb, amsthm, enumerate, fullpage, bm}
\usepackage{verbatim, enumitem, bbm, etoolbox}
\usepackage{latexsym}

\newcommand{\std}{{\mbox{\scriptsize{std}}}}

\apptocmd{\lim}{\limits}{}{}

\newcommand{\acl}{\textrm{acl}}

\newcommand{\model}{\models}

\theoremstyle{definition}
\newtheorem{thm}{Theorem}[section]
\newtheorem{theorem}[thm]{Theorem}

\newtheorem{lemma}[thm]{Lemma}

\numberwithin{subcase}{case}

\theoremstyle{definition}
\newtheorem{definition}[thm]{Definition}
\newtheorem{corollary}[thm]{Corollary}
\newtheorem{remark}[thm]{Remark}

\def\forkindep{\mathrel{\raise0.2ex\hbox{\ooalign{\hidewidth$\vert$\hidewidth\cr\raise-0.9ex\hbox{$\smile$}}}}}

\def\Ind{\setbox0=\hbox{$x$}\kern\wd0\hbox to 0pt{\hss$\mid$\hss}
	\lower.9\ht0\hbox to 0pt{\hss$\smile$\hss}\kern\wd0}

\def\Notind{\setbox0=\hbox{$x$}\kern\wd0\hbox to 0pt{\mathchardef
		\nn=12854\hss$\nn$\kern1.4\wd0\hss}\hbox to 0pt{\hss$\mid$\hss}\lower.9\ht0
	\hbox to 0pt{\hss$\smile$\hss}\kern\wd0}

\def\phi{\varphi}

\def\<{\langle}
\def\>{\rangle}
\def\acl{\textrm{acl}}

\makeatletter
\def\blfootnote{\xdef\@thefnmark{}\@footnotetext}
\makeatother

\begin{document}	

	\bibliographystyle{plain}
	
	\author{Douglas Ulrich\!\!\
	\thanks{Partially supported
by Laskowski's NSF grant DMS-1308546.}\\
Department of Mathematics\\University of California, Irvine}
	\title{Cardinal Characteristics of Models of Set Theory}
	\date{\today} 
	
	
	\maketitle
	
	
\begin{abstract}
We continue our investigation from \cite{InterpOrdersUlrich} of Shelah's interpretability orders $\trianglelefteq^*_\kappa$ as well as the new orders $\trianglelefteq^\times_\kappa$. In particular, we give streamlined proofs of the existence of minimal unstable, unsimple and nonlow theories in these orders, and we give a similar analysis of the hypergraph examples $T_{n, k}$ of Hrushovski \cite{Hrush}. Using the technology of \cite{BVModelsUlrich}, we prove that if  $\mathcal{B}$ is a complete Boolean algebra with the $\lambda$-c.c., then no nonprincipal ultrafilter on $\mathcal{U}$ $\lambda^+$-saturates any unsimple theory.
\end{abstract}

\section{Introduction}

In \cite{SH500}, Shelah introduced the interpretabillity orders $\trianglelefteq^*_\kappa$ as abstractions of Keisler's order $\trianglelefteq$. In \cite{InterpOrdersUlrich}, we cast these orders in terms of cardinal characteristics of nonstandard models of $ZFC^-$. Furthermore, we introduced new interpretability orders $\trianglelefteq^\times_\kappa$, which refine $\trianglelefteq^*_\kappa$ and which smooth out many technical difficulties. In this paper, we demonstrate the power of these notions in several applications. 

To begin, we review the setup of \cite{InterpOrdersUlrich}.

\begin{definition} $ZFC^-$ is $ZFC$ without powerset, and with replacement strengthened to collection, and with choice strengthened to the well-ordering principle; this is as in \cite{ZFCminus}.

	Say that $\hat{V} \models ZFC^-$ is an $\omega$-model, or is $\omega$-standard, if every natural number of $\hat{V}$ is standard (i.e. has finitely many $\hat{\in}$-predecessors).
	
	$V$ will denote a transitive model of $ZFC^-$. $\hat{V}$ will denote an $\omega$-nonstandard model of $ZFC^-$. Frequently $\hat{V}$ will come from an embedding $\mathbf{j}:V \preceq \hat{V}$, where $V$ is transitive. Whenever $\hat{V} \models ZFC^-$, we will identify $HF$ (the hereditarily finite sets) with its copy in $\hat{V}$. Other elements of $\hat{V}$ will usually be decorated with a hat, for instance we write $\hat{\omega}$ rather than $(\omega)^{\hat{V}}$. Given $X \subseteq \hat{V}$, we say that $X$ is an internal subset of $\hat{V}$ if there is some $\hat{X} \in \hat{V}$ such that $X = \{\hat{y} \in \hat{V}: \hat{y} \hat{\in} \hat{X}\}$. In this case, we usually identify $X$ with $\hat{X}$ and will write that $X \in \hat{V}$. 
	
	Suppose $V \models ZFC^-$ is transitive, and $\mathbf{j}: V \preceq \hat{V}$. Say that $X \subseteq \hat{V}$ is pseudofinite (with respect to $\hat{V}$) if there is some $\hat{X} \in \hat{V}$ finite in the sense of $\hat{V}$, with $X \subseteq \hat{X}$. So if $\hat{X} \in \hat{V}$, then $\hat{X}$ is pseudofinite if and only if it is finite in the sense of $\hat{V}$.
	
	Suppose $M \in V$ is an $\mathcal{L}$-structure (it follows that $\mathcal{L} \in V$). Note that $\mathbf{j}(M)$ is a $\mathbf{j}(\mathcal{L})$-structure, where possibly some of the symbols of $\mathbf{j}(\mathcal{L})$ are nonstandard; let $\mathbf{j}_{\std}(M)$ be the ``reduct" to $\mathcal{L}$.  Say that $\mathbf{j}_{\std}(M)$ is $\kappa$-pseudosaturated if for every pseudofinite $A \subseteq \mathbf{j}_{\std}(M)$ with $|A| < \kappa$, and for every $n < \omega$: every type $p(\overline{x}) \in S^n(A)$ is realized in $\mathbf{j}_{\std}(M)$. (It is enough to take $n=1$.)
	\end{definition}

\vspace{2 mm}

\noindent \textbf{Convention.} We operate entirely in $ZFC$; thus everything is a set, including formulas. Whenever $T$ is a complete countable theory, we suppose $T$ comes equipped with an an injection from the symbols of $T$ into $\omega$. In particular, whenever $T \in V \models ZFC^-$, $T$ is countable in $V$.

\vspace{1 mm}

\begin{remark}
As an exercise in terminology (Lemma 3.10 of \cite{InterpOrdersUlrich}), note that when $\kappa > \aleph_0$, we have that $\mathbf{j}_{\std}(M)$ is $\kappa$-pseudosaturated if and only if every pseudofinite partial type over $\mathbf{j}_{\std}(M)$ of cardinality less than $\kappa$ is realized in $\mathbf{j}_{\std}(M)$. To parse this: a partial type over $\mathbf{j}_{\std}(M)$ is pseudofinite if it is contained in some set $\hat{X} \in \hat{V}$ which is finite in the sense of $\hat{V}$; we can suppose by the Separation schema that $\hat{X}$ is a finite set of $\mathbf{j}(\mathcal{L})$-formulas over $\mathbf{j}(M)$, but we cannot arrange for all the formulas in $\hat{X}$ to be standard (they may involve nonstandard symbols, and may be of nonstandard length), and the crux of the issue is whether we can arrange that $\hat{V} \models ``\hat{X}$ is consistent$"$.
\end{remark}

The reader may take the following as the definition of $\trianglelefteq^*_{\lambda \kappa}$, for our purposes; this is Lemma 3.6 of \cite{InterpOrdersUlrich}. Following \cite{InterpOrders}, we consider every structure to be $1$-saturated.

\begin{lemma}\label{SatAndPseudo0}
	Suppose $\lambda$ is an infinite cardinal, $\kappa$ is an infinite cardinal or $1$, and $T_0, T_1$ are complete countable theories. Then the following are equivalent:
	\begin{itemize}
		\item[(A)] $T_0 \trianglelefteq^*_{\lambda \kappa} T_1$;
		\item[(B)] There is some countable transitive $V \models ZFC^-$ with $T_0, T_1 \in V$, and some $M_i \models T_i$ both in $V$, such that whenever $\mathbf{j}:V \preceq \hat{V}$, if $\hat{V}$ is $\kappa$-saturated and $\omega$-nonstandard and if $\mathbf{j}_{\std}(M_1)$ is $\lambda^+$-saturated, then $\mathbf{j}_{\std}(M_0)$ is $\lambda^+$-saturated.
	\end{itemize}
\end{lemma}

We say that $T_0 \trianglelefteq^*_\kappa T_1$ if $T_0 \trianglelefteq^*_{\lambda \kappa} T_1$ for all $\lambda$. Note that we follow the original indexing system of Shelah \cite{SH500}, and so $\trianglelefteq^*_{\aleph_1}$ refines Keisler's order $\trianglelefteq$.

It is a serious annoyance that the choice of $M_i \in V$ matters, and indeed is often delicate. The following remedy is the conjunction of Theorem 3.13 ($\kappa> \aleph_0$) and Corollary 7.7 ($\kappa = \aleph_0$) from \cite{InterpOrdersUlrich}.

\begin{theorem}\label{ModKeislerOrder}
	Suppose $V \models ZFC^-$ is transitive and $\mathbf{j}: V \preceq \hat{V}$ is $\omega$-nonstandard, and $T \in V$ is a complete countable theory, and $M_0, M_1 \in V$ are two models of $T$. Suppose $\kappa$ is an infinite cardinal. Then $\mathbf{j}_{\std}(M_0)$ is $\kappa$-pseudosaturated if and only if $\mathbf{j}_{\std}(M_1)$ is.
\end{theorem}

This gives a new way of viewing Keisler's fundamental theorem on saturation of ultrapowers, since if $\mathcal{U}$ is an ultrafilter on $\mathcal{P}(\lambda)$, and if we let $\mathbf{j}: V \preceq \hat{V} := V^\lambda/\mathcal{U}$ be the {\L}o{\'s} embedding, then for every $M \in V$, $M^\lambda/\mathcal{U} \cong \mathbf{j}_{\std}(M)$. Further, if $\mathcal{U}$ is $\lambda$-regular then every subset of $V^\lambda/\mathcal{U}$ of size at most $\lambda$ is pseudofinite, and hence $\mathbf{j}_{\std}(M)$ is $\lambda^+$-saturated if and only if it is $\lambda^+$-pseudosaturated.

Theorem~\ref{ModKeislerOrder} suggests the following tweak to the interpretability orders:

\begin{definition}
	Suppose $V \models ZFC^-$ is transitive, $\mathbf{j}: V \preceq \hat{V}$ is $\omega$-nonstandard, and suppose $T$ is a complete countable theory with $T \in V$. Suppose $\kappa$ is an infinite cardinal. Then say that $\hat{V}$ $\kappa$-pseudosaturates $T$ if for some or every $M \models T$ with $M \in V$, $\mathbf{j}_{\std}(M)$ is $\kappa$-pseudosaturated. 
	
	Suppose $\lambda$ is infinite and $\kappa$ is infinite or $1$. Then say that $T_0 \trianglelefteq^\times_{\lambda \kappa} T_1$ if there is some countable transitive $V \models ZFC^-$ containing $T_0, T_1$ such that whenever $\mathbf{j}: V \preceq \hat{V}$, if $\hat{V}$ is $\kappa$-saturated and $\omega$-nonstandard, and if $\hat{V}$ $\lambda^+$-pseudosaturates $T_1$, then also it $\lambda^+$-pseudosaturates $T_0$. Say that $T_0 \trianglelefteq^\times_\kappa T_1$ if $T_0 \trianglelefteq^\times_{\lambda \kappa} T_1$ for all infinite $\lambda$.
\end{definition}

Theorem 10.6 of \cite{InterpOrdersUlrich} states that $\trianglelefteq^\times_\kappa \subseteq \trianglelefteq^*_\kappa$, in other words: to prove $T_0 \trianglelefteq^*_\kappa T_1$ it is enough to show $T_0 \trianglelefteq^\times_\kappa T_1$. In practice, this turns out to be much cleaner. In any case, we suspect that $\trianglelefteq^*_\kappa = \trianglelefteq^\times_\kappa$.

In Section~\ref{SurveyLocalSec}, we lift Malliaris's theorem that Keisler's order is local \cite{PhiTypes} to the context of $\trianglelefteq^\times_{\aleph_1}$, and introduce the notion of patterns. Patterns have been studied under different notation by Shelah \cite{SH702} and then Malliaris \cite{CharSequence}, although Malliaris was the first to connect them to Keisler's order.

In Sections~\ref{RandGraphSec}, \ref{StarRandGraphSec} and \ref{TCasSec}, we use this machinery to give streamlined proofs of the existence of minimal unstable, unsimple and nonlow theories in $\trianglelefteq^\times_1$ (which are thus minimal in all of the other orders $\trianglelefteq^\times_\kappa, \trianglelefteq^*_\kappa$ and $\trianglelefteq$ as well). In particular, we prove the following, where $T_{rg}$ is the theory of the random graph, and $T_{nlow}$ is the supersimple nonlow theory introduced by Casanovas and Kim \cite{SupersimpleNonlow}, and where $T_{rf}$ is the theory of the random binary function.\footnote{We will define low in Section~\ref{Prelim}, but for now, note---there are two definitions of ``low" in circulation. We follow the original definition of Buechler \cite{Buechler}, so in particular low implies simple.}

\begin{theorem}\label{SummaryTheorem}
Suppose $V \models ZFC^-$ is transitive, and $\mathbf{j}: V \preceq \hat{V}$ is $\omega$-nonstandard, and $\lambda$ is an infinite cardinal, and $T \in V$ is a complete countable theory. Then:

\begin{itemize} 
	\item[(A)] If $\hat{V}$ $\lambda^+$-pseudosaturates $T$ and $T$ is unstable, then for all disjoint, pseudofinite $X_0, X_1 \subseteq \hat{V}$ with $|X_i| \leq \lambda$, there exist disjoint $\hat{X}_0, \hat{X}_1 \in \hat{V}$ with $X_0 \subseteq \hat{X}_0$ and $X_1 \subseteq \hat{X}_1$. If $T = T_{rg}$ then the converse holds.
	\item[(B)] If $\hat{V}$ $\lambda^+$-pseudosaturates $T$ and $T$ is unsimple, then every pseudofinite partial function from $\hat{V}$ to $\hat{V}$ of cardinality at most $\lambda$ can be extended to some internal partial function $\hat{f}$ from $\hat{V}$ to $\hat{V}$. If $T = T_{rf}$ then the converse holds.
	
	\item[(C)] If $\hat{V}$ $\lambda^+$-pseudosaturates $T$ and $T$ is nonlow, then the conclusion in (A) holds, and further: for every $X \subseteq \hat{V}$ with $|X| \leq \lambda$, and for every $\hat{n} < \hat{\omega}$ nonstandard, there exists some $\hat{X} \in \hat{V}$ such that $\hat{V} \models ``|\hat{X}| = \hat{n}"$ and such that $X \subseteq \hat{X}$. If $T= T_{nlow}$ then the converse holds.
	\end{itemize}
\end{theorem}

We can cast this theorem more systematically as follows.

\begin{definition}
	Suppose $V \models ZFC^-$ is transitive, and $T \in V$ is a complete countable theory, and $\mathbf{j}:V \preceq \hat{V}$ is $\omega$-nonstandard. Then let $\lambda_{\hat{V}}(T)$ be the least infinite cardinal such that $\hat{V}$ does not $\lambda_{\hat{V}}(T)^+$-pseudosaturate $T$; possibly $\lambda_{\hat{V}}(T) = \infty$ (if $\hat{V}$ $\lambda$-pseudosaturates $T$ for all $\lambda$).
\end{definition}

In other words, each complete countable theory $T$ induces a cardinal characteristic of models of $ZFC^-$; we are interested in determining which of these cardinal characteristics can be separated. When $\lambda_{\hat{V}}(T) > \aleph_0$, then it follows from the definition that $\hat{V}$ does $\lambda_{\hat{V}}(T)$-pseudosaturate $T$; when $\lambda_{\hat{V}}(T) = \aleph_0$, then $\hat{V}$ may or may not $\aleph_0$-pseudosaturate $T$, but the situation is understood (see Corollary 7.7 of \cite{InterpOrdersUlrich}). 
\vspace{1 mm}

\noindent \textbf{Question.} Suppose $T$ is a complete countable theory. What is the function $\hat{V} \mapsto \lambda_{\hat{V}}(T)$?

\vspace{1 mm}

To reprhase, Theorem~\ref{SummaryTheorem} determines $\lambda_{\hat{V}}(T_{rg})$, $\lambda_{\hat{V}}(T_{rf})$ and $\lambda_{\hat{V}}(T_{nlow})$, and shows that these are the maximal possible values for unstable, unsimple and nonlow theories, respectively.

%
%
%
%
%

In Section~\ref{KeislerNewTNK}, we give a similar treatment of the hypergraph examples $T_{n, k}$. Namely, for $n > k \geq 2$, let $T_{n, k}$ be the random $k$-ary $n$-clique free hypergraph. These are due to Hrushovksi \cite{Hrush}; Malliaris and Shelah use them in \cite{InfManyClass} to prove that Keisler's order has infinitely many classes.

In Section~\ref{FullBVModelsSec} we recall the setup of full Boolean-valued models from \cite{BVModelsUlrich}, and recast our results in terms of ultrafilters. In particular, for every ultrafilter $\mathcal{U}$ on a complete Boolean algebra $\mathcal{B}$, and for every complete countable theory $T$, we define what it means for $\mathcal{U}$ to $\lambda^+$-saturate $T$. Keisler's order can be framed as follows: $T_0 \trianglelefteq_\lambda T_1$ if and only if for every complete Boolean algebra $\mathcal{B}$ with the $\lambda^+$-c.c. and for every ultrafilter $\mathcal{U}$ on $\mathcal{B}$, if $\mathcal{U}$ $\lambda^+$-saturates $T_1$, then $\mathcal{U}$ $\lambda^+$-saturates $T_0$; and $T_0 \trianglelefteq T_1$ if and only if $T \trianglelefteq_\lambda T_1$ for all $\lambda$. 

In Section~\ref{SurveyCCSec} we tie off several strands of non-saturation arguments. To give the reader context, we quote the following theorems. (B) is due to Malliaris and Shelah  \cite{Optimals}. I prove (C) and (D) in \cite{LowDividingLine}, and I prove (A) in \cite{BVModelsUlrich}. As notation, if $\mathcal{B}$ is a complete Boolean algebra, then $\mathcal{B}$ has the $\kappa$-c.c. (chain condition) if $\mathcal{B}$ has an antichain of size $\lambda$, and $\mbox{c.c.}(\mathcal{B})$ is the least $\kappa$ for which this fails, i.e. the least cardinality $\kappa$ for which $\mathcal{B}$ has no antichain of size $\kappa$.

\begin{theorem}\label{OptimalsTheorem} The following are all true.
	
	\begin{itemize}
		
		\item[(A)] Suppose $\lambda$ is an infinite cardinal, and $\mathcal{B}$ is a complete Boolean algebra with an antichain of size $\lambda$ (i.e. failing the $\lambda$-c.c.). Then there is a $\lambda^+$-good ultrafilter $\mathcal{U}$ on $\mathcal{B}$ (i.e. an ultrafilter $\mathcal{U}$ which $\lambda^+$-saturates every complete countable theory $T$).
		\item[(B)] If there is a supercompact cardinal $\lambda$, then there is a complete Boolean algebra $\mathcal{B}$ with the $\lambda$-c.c. and an ultrafilter $\mathcal{U}$ on $\mathcal{B}$, such that $\mathcal{U}$ $\lambda^+$-saturates exactly the simple theories.
		\item[(C)] There is some complete Boolean algebra $\mathcal{B}$ and some $\aleph_1$-incomplete ultrafilter $\mathcal{U}$ on $\mathcal{B}$, such that $\mathcal{U}$ $\mbox{c.c.}(\mathcal{B})^+$-saturates exactly the low theories.
		
		\item[(D)] Suppose $\mathcal{B}$ is a complete Boolean algebra, and $\mathcal{U}$ is an $\aleph_1$-incomplete ultrafilter on $\mathcal{B}$. Then $\mathcal{U}$ does not $\mbox{c.c.}(\mathcal{B})^+$-saturate any nonlow theory.

	\end{itemize}
\end{theorem}

We complete the picture with the following. Here, we say that the ultrafilter $\mathcal{U}$ on $\mathcal{B}$ is principal if $\bigwedge \mathcal{U}$ is nonzero (this coincides with the usual definition if $\mathcal{B} =\mathcal{P}(\lambda)$.)

\begin{theorem}\label{ExistenceConverseFirst}
	Suppose $\mathcal{B}$ is a complete Boolean algebra and $\mathcal{U}$ is a nonprincipal ultrafilter on $\mathcal{B}$. Then $\mathcal{U}$ does not $\mbox{c.c.}(\mathcal{B})^+$-saturate any unsimple theory. 
\end{theorem}

We also give a short proof of (D) above, making use of our new machinery.

\vspace{1 mm}

\noindent \textbf{Acknowledgements.} We would like to thank Vincent Guingona, Alexei Kolesnikov, Chris Laskowski, Pierre Simon and Jindrich Zapletal for several helpful conversations.

\section{Preliminaries}\label{Prelim}

We collect together some facts and terminology we will need.

\subsection{Model-Theoretic Preliminaries}
\begin{definition}
	Suppose $T$ is a complete countable theory and $\phi(\overline{x}, \overline{y})$ is a formula; for convenience, we write it as $\phi(x, y)$. Then: 
	
	\begin{itemize}
		\item $\phi(x, y)$ has the independence property ($IP$) if there are $(b_n: n < \omega)$ such that for all disjoint $u, v \subseteq \omega$, $\{\phi(x, b_n: n \in )\} \cup \{\lnot \phi(x, b_n): n \in v\}$ is consistent. Otherwise, $\phi(x, y)$ has $NIP$.
		\item $\phi(x, y)$ has the strict order property of the second kind ($SOP_2$) if there are $(b_s: s \in \omega^{<\omega})$, such that for each $\eta \in \omega^\omega$, $(\phi(x, b_{\eta \restriction_n}): n < \omega)$ is consistent, but whenever $s, t \in \omega^{<\omega}$ are incomparable, $\phi(x, b_s) \land \phi(x, b_t)$ is inconsistent. 
		\item $\phi(x, y)$ has the tree property of the second kind $(TP_2)$ if there are $(b_{n, m}: n, m < \omega)$ such that for all $n < \omega$ and for all $m < m' < \omega$, $\exists \overline{x}( \phi(x, b_{n, m}) \wedge \phi(x, b_{n, m'}))$ is inconsistent, but such that for all $\eta \in \omega^\omega$, $\{\phi(x, b_{n, \eta(n)})\}$ is consistent. Otherwise $\phi(x, y)$ has $NTP_2$.
		\item  $\phi(x, y)$ has the finite dividing property if for every $k$ there is some indiscernible sequence $(b_n: n < \omega)$ over the empty set such that $\{\phi(x, b_n): n < \omega\}$ is $k$-consistent but not consistent.
	\end{itemize}
\end{definition}

\begin{remark}
	The tree property of the first kind $TP_1$ is equivalent to the strict order property of the second kind $SOP_2$; we pick the term $SOP_2$ to use. See \cite{TP1} for a comparison. 
\end{remark}

We recall that in simple theories, forking (equivalently dividing) is a well-behaved independence relation.

\begin{definition}
	Suppose $T$ is a complete countable theory. Then $T$ is low if $T$ is simple and does not have the finite dividing property.
\end{definition}

\begin{remark}
There are multiple definitions of low in use. Our definition is equivalent to the original definition of Buechler \cite{Buechler}, and is also how the low is defined in \cite{KimForking}, for instance. In other places in the literature, the hypothesis that $T$ is simple is dropped, e.g. as in \cite{DividingLine}.
\end{remark}

We recall two dichotomy theorems of Shelah. The following is Theorem 0.2 of \cite{ShelahSimple}:

\begin{theorem}\label{NonSimpleDichotomy}
	$T$ is unsimple if and only if either $T$ has $TP_2$ or else $T$ has $SOP_2$. 
\end{theorem}

The following is also well-known:

\begin{theorem}\label{NonStableDichotomy}
	$T$ is unstable if and only if either $T$ has $IP$ or $SOP_2$.
\end{theorem}
\begin{proof}
	Theorem II.4.7 of \cite{ShelahIso} states that $T$ is unstable if and only if either $T$ has $IP$ or else $T$ has $SOP$. But if $T$ has $SOP$ then $T$ has $SOP_2$, and if $T$ has $SOP_2$ then $T$ is unstable, so the theorem follows.
\end{proof}

\subsection{Pseudosaturation}

We recall some facts from \cite{InterpOrdersUlrich}.

The following is a key cardinal characteristic of models of set theory; in the context of ultrapower embeddings, the definition is due to Malliaris and Shelah \cite{pEqualsTref}.

\begin{definition}
	
	Suppose $(L, <)$ is a linear order with proper initial segment $\omega$. If $\kappa, \theta$ are infinite regular cardinals, then a $(\kappa, \theta)$-pre-cut in $L$ is a pair of sequences $(\overline{a}, \overline{b}) = (a_\alpha: \alpha < \kappa)$, $(b_\beta: \beta < \theta)$ from $L$, such that for all $\alpha < \alpha'$, $\beta < \beta'$, we have $a_\alpha < a_{\alpha'} < b_{\beta'} < b_{\beta}$. $(\overline{a}, \overline{b})$ is a cut if there is no $c \in L$ with $a_\alpha < c < b_\beta$ for all $\alpha, \beta$. Let the cut spectrum of $(L, <)$ be $\mathcal{C}(L, <) := \{(\kappa, \theta): L \mbox{ admits a } (\kappa, \theta) \mbox{ cut}\}$. Define $\mbox{cut}(L, <) = \mbox{min}\{\kappa + \theta:  (\kappa, \theta) \in \mathcal{C}(L, <)\}$.
	
		Suppose $\hat{V} \models ZFC^-$ is nonstandard. Define $\mathcal{C}_{\hat{V}}= \mathcal{C}(\hat{\omega}, \hat{<})$, and define  $\mathfrak{p}_{\hat{V}} = \mbox{cut}(\hat{\omega}, \hat{<})$.
\end{definition}

$\mathfrak{p}_{\hat{V}}$ is the smallest cardinal characteristic of models of set theory that is relevant for pseudosaturation. The following two theorems are mild generalizations of results of Malliaris and Shelah \cite{pEqualsTref} \cite{pEqualsT2} to the context of $\trianglelefteq^\times_1$. The first is Theorem 5.3 in \cite{InterpOrdersUlrich}, the second is the conjunction of Corollary 5.4 and Theorem 5.7.

\begin{theorem}\label{localSaturation2}
	Suppose $\hat{V} \models ZFC^-$ is $\omega$-nonstandard and $\mathfrak{p}_{\hat{V}} \geq \aleph_1$. Suppose $p(x)  = \{\phi_\alpha(x, \hat{a}_\alpha): \alpha < \lambda\}$ is a type over $\hat{V}$ of cardinality $\lambda < \mathfrak{p}_{\hat{V}}$, and suppose $\{\hat{a}_\alpha: \alpha < \lambda\}$ is pseudofinite. Then $p(x)$ is realized in $\hat{V}$, provided either of the following conditions are met.
	
	\begin{itemize}
		\item[(A)] There is some $n < \omega$ such that each $\phi_\alpha(x, a_\alpha)$ is $\Sigma_n$.
		
		\item[(B)] Every countable subset of $\hat{V}$ is pseudofinite.
	\end{itemize}
\end{theorem}

\begin{theorem}\label{SOP2max}
	Suppose $V \models ZFC^-$ is transitive, $T \in V$ is a complete countable theory, and $\mathbf{j}: V \preceq \hat{V}$ is $\omega$-nonstandard. Then $\lambda_{\hat{V}}(T) \geq \mathfrak{p}_{\hat{V}}$. If $T$ has $SOP_2$ then equality is attained.
\end{theorem}

\subsection{Full Boolean-Valued Models}

We recall the setup of \cite{BVModelsUlrich}. This won't be used until Section~\ref{FullBVModelsSec}, and the reader may wish to defer reading this subsection until then.

As a convention, if $X$ is a set and $\mathcal{L}$ is a language, then $\mathcal{L}(X)$ is the set of formulas of $\mathcal{L}$ with parameters taken from $X$.

Suppose $\mathcal{B}$ is a complete Boolean algebra. A $\mathcal{B}$-valued structure is a pair $(\mathbf{M}, \| \cdot \|_{\mathbf{M}})$ where:

\begin{enumerate}
	\item $\mathbf{M}$ is a set;
	\item  $\phi \mapsto \|\phi\|_{\mathbf{M}}$ is a map from $\mathcal{L}( \mathbf{M})$ to $\mathcal{B}$;
	\item If $\phi$ is a logically valid sentence then $\|\phi\|_{\mathbf{M}} = 1$;
	\item For every formula $\phi \in \mathcal{L}( \mathbf{M})$, we have that $\|\lnot \phi\|_{\mathbf{M}}= \lnot\|\phi\|_{\mathbf{M}}$;
	\item For all $\phi, \psi$, we have that $\|\phi \land \psi\|_{\mathbf{M}} = \|\phi\|_{\mathbf{M}} \land \|\psi\|_{\mathbf{M}}$;
	\item For every formula $\phi(x)$ with parameters from $\mathbf{M}$, $\|\exists x \phi(x)\|_{\mathbf{M}} = \bigvee_{a \in \mathbf{M}} \|\phi(a)\|_{\mathbf{M}}$;
	\item For all $a, b \in \mathbf{M}$ distinct, $\|a = b\|_{\mathbf{M}} < 1$.
\end{enumerate}

We are only interested in the case when $\mathbf{M}$ is full, i.e. when in fact  $\|\exists x \phi (x, \overline{a}) \|_{\mathbf{M}} = \mbox{max}_{a \in \mathbf{M}} \|\phi(a, \overline{a})\|_{\mathbf{M}}$. If $T$ is a theory, then we write $\mathbf{M} \models^{\mathcal{B}} T$, and say that $\mathbf{M}$ is a full $\mathcal{B}$-valued model of $T$, if $\|\phi\|_{\mathbf{M}} = 1$ for all $\phi\in T$.

For example, (ordinary) $\mathcal{L}$-structures are the same as full $\{0, 1\}$-valued $\mathcal{L}$-structures, which can thus be viewed as full $\mathcal{B}$-valued structures for any $\mathcal{B}$. Also, if $M$ is an $\mathcal{L}$-structure and $\lambda$ is a cardinal, then $M^\lambda$ is a $\mathcal{P}(\lambda)$-valued $\mathcal{L}$-structure; moreover, we have the canonical elementary embedding $\mathbf{i}: M \preceq M^\lambda$, given by the diagonal map. We call this the pre-{\L}o{\'s} embedding. More generally, for any complete Boolean algebra $\mathcal{B}$ we can define the $\mathcal{B}$-valued structure $M^{\mathcal{B}}$.

If $\mathbf{M}$ is a full $\mathcal{B}$-valued model of $T$ and $\mathcal{U}$ is an ultrafilter on $\mathcal{B}$, then we can form the specialization $\mathbf{M}/\mathcal{U} \models T$, which comes equipped with a canonical surjection $[\cdot]_{\mathcal{U}}: \mathbf{M}\to \mathbf{M}/\mathcal{U}$, satisfying that for all $\phi(\overline{a}) \in \mathcal{L}(\mathbf{M})$, $\mathbf{M}/\mathcal{U}\models \phi([\overline{a}]_{\mathcal{U}})$ if and only if $\|\phi(\overline{a})\|_{\mathbf{M}} \in \mathcal{U}$. This generalizes the ultrapower construction $M^\lambda/\mathcal{U}$; note that the {\L}o{\'s} embedding of $M$ into $M^\lambda/\mathcal{U}$ is the composition of the pre-{\L}o{\'s} embedding with $[\cdot]_{\mathcal{U}}$.

In \cite{BVModelsUlrich}, we prove the following compactness theorem for  full Boolean-valued models:

\begin{theorem}\label{Compactness}
	Suppose $\mathcal{B}$ is a complete Boolean algebra, $X$ is a set, $\Gamma \subseteq \mathcal{L}( X)$, and $F_0, F_1: \Gamma \to  \mathcal{B}$ with $F_0(\phi(\overline{a})) \leq F_1(\phi(\overline{a}))$ for all $\phi(\overline{a}) \in \Gamma$. Then the following are equivalent:
	
	\begin{itemize}
		\item[(A)] There is some full $\mathcal{B}$-valued structure $\mathbf{M}$ and some map $\tau: X \to \mathbf{M}$, such that for all $\phi(\overline{a}) \in \Gamma$, $F_0(\phi(\overline{a})) \leq \|\phi(\tau(\overline{a}))\|_{\mathbf{M}} \leq F_1(\phi(\overline{a}))$;
		
		\item[(B)] For every finite $\Gamma_0 \subseteq \Gamma$ and for every $\mathbf{c} \in \mathcal{B}_+$, there is some $\{0, 1\}$-valued $\mathcal{L}$-structure $M$ and some map $\tau: X \to M$, such that for every $\phi(\overline{a}) \in \Gamma$, if $\mathbf{c} \leq F_0(\phi(\overline{a}))$ then $M \models \phi(\tau(\overline{a}))$, and if $\mathbf{c} \leq \lnot F_1(\phi(\overline{a}))$ then $M \models \lnot \phi(\tau(\overline{a}))$. 
	\end{itemize}
\end{theorem}

Here is a first application: given $\mathcal{B}$-valued models $\mathbf{M}\subseteq \mathbf{N}$, say that $\mathbf{M} \preceq \mathbf{N}$ if $\|\cdot\|_{\mathbf{M}} \subseteq \| \cdot \|_{\mathbf{N}}$. Say that $\mathbf{N}$ is $\lambda^+$-saturated if for every $\mathbf{M}_0 \preceq \mathbf{N}$ with $|\mathbf{M}_0| \leq \lambda$ and for every $\mathbf{M}_1 \succeq \mathbf{M}_0$ with $|\mathbf{M}_1| \leq \lambda$, there is some elementary embedding $f: \mathbf{M}_1 \preceq \mathbf{N}$ extending the inclusion from $\mathbf{M}_0$ into $\mathbf{N}$. Then in \cite{BVModelsUlrich}, we show that for every $\mathcal{B}$-valued structure $\mathbf{M}$ and for every $\lambda$, there is an elementary extension $\mathbf{N} \succeq \mathbf{M}$ such that $\mathbf{N}$ is full and moreover $\lambda^+$-saturated.

Suppose $T$ is a complete countable theory, and $\mathcal{U}$ is an ultrafilter on the complete Boolean algebra $\mathcal{B}$. We observe in \cite{BVModelsUlrich} that if there is some $\lambda^+$-saturated $\mathbf{M} \models^{\mathcal{B}} T$ with $\mathbf{M}/\mathcal{U}$ $\lambda^+$-saturated, then for every $\lambda^+$-saturated $\mathbf{M} \models^{\mathcal{B}} T$, $\mathbf{M}/\mathcal{U}$ is $\lambda^+$-saturated. We define that $\mathcal{U}$ $\lambda^+$-saturates $T$ in this case. 

Finally, in \cite{BVModelsUlrich} we give the following convenient characterization of Keisler's order:

\begin{theorem}\label{KeislerChar}
	Suppose $T_0, T_1$ are theories. Then $T_0 \trianglelefteq T_1$ if and only if for every $\lambda$, for every complete Boolean algebra $\mathcal{B}$ with the $\lambda^+$-c.c., and for every ultrafilter $\mathcal{U}$ on $\mathcal{B}$, if $\mathcal{U}$ $\lambda^+$-saturates $T_1$, then $\mathcal{U}$ $\lambda^+$-saturates $T_0$.
\end{theorem}

\noindent \textbf{Convention.} $\mathbf{V}$ will denote a full $\mathcal{B}$-valued model of $ZFC^-$ for some complete Boolean algebra $\mathcal{B}$, often associated with an elementary embedding $\mathbf{i}: V \preceq \mathbf{V}$ for some transitive $V$ (for example, if $\mathbf{V} = V^{\mathcal{B}}$ is the Boolean ultrapower, then $\mathbf{i}$ would be the pre-{\L}o{\'s} embedding).

\begin{definition}
	Suppose $\mathcal{B}$ is a complete Boolean algebra, $V \models ZFC^-$ is transitive, and $\mathbf{i}: V \preceq \mathbf{V} \models^{\mathcal{B}} ZFC^-$. Given $X \subseteq V$, let $\mathbf{i}_{\std}(X) = \{\mathbf{a}\in V: \|\mathbf{a} \in \mathbf{i}(X)\|_{\mathcal{B}}\} = 1$. Suppose $M \in V$ is a structure in the countable language $\mathcal{L}$, with domain $\mbox{dom}(M)$. Then let $\mathbf{i}_{\std}(M)$ be the full $\mathcal{B}$-valued $\mathcal{L}$-structure defined as follows. Its domain is $\mathbf{i}_{\std}(\mbox{dom}(M))$. Given a formula $\phi(\mathbf{a}_i: i < n) \in \mathcal{L}(\mathbf{i}_{\std}(\mbox{dom}(M)))$, let $\|\phi(\mathbf{a}_i: i < n)\|_{\mathbf{i}_{\std}(M)} = \|\mathbf{i}(M) \models \phi(\mathbf{a}_i: i < n)\|_{\mathbf{V}}$. (Note that in practice, we usually denote $M$ and $\mbox{dom}(M)$ by the same symbol $M$.)
\end{definition}

The following corollaries are proven in \cite{InterpOrdersUlrich}.

\begin{corollary}\label{JustifyDefOfSat2}
	Suppose $\mathcal{B}$ is a complete Boolean algebra, $\mathcal{U}$ is an ultrafilter on $\mathcal{B}$, $\lambda$ is a cardinal, and $T$ is a complete countable theory. Then the following are equivalent:
	
	\begin{itemize}
		\item[(A)] $\mathcal{U}$ $\lambda^+$-saturates $T$.
		\item[(B)] For some or every transitive $V \models ZFC^-$, and for some or every $\mathbf{i}: V \preceq \mathbf{V}$ with $\mathbf{V}$ $\lambda^+$-saturated, and for some or every $M \models T$ with $M \in V$, $\mathbf{i}_{\std}(M)/\mathcal{U}$ is $\lambda^+$-saturated.
	\end{itemize}
\end{corollary}

\begin{corollary}\label{KeislerAndInterp2}
	Suppose $\mathcal{B}$ is a complete Boolean algebra, $\mathcal{U}$ is an ultrafilter on $\mathcal{B}$, $\lambda$ is a cardinal, and $T$ is a complete countable theory. Then the following are equivalent:
	
	\begin{itemize}
		\item[(A)] $\mathcal{U}$ $\lambda^+$-saturates $T$.
		\item[(B)] For some or every transitive $V \models ZFC^-$ with $T \in V$, and for some or every $\mathbf{i}: V \preceq \mathbf{V}$ with $\mathbf{V}$ $\lambda^+$-saturated, and for some or every $M \models T$ with $M \in V$, $\mathbf{i}_{\std}(M)/\mathcal{U}$ is $\lambda^+$-saturated.
		\item[(C)] For some or every transitive $V \models ZFC^-$, for some or every $\mathbf{i}: V \preceq \mathbf{V}$ with $\mathbf{V}$ $\lambda^+$-saturated, $\mathbf{V}/\mathcal{U}$ $\lambda^+$-pseudosaturates $T$.
	\end{itemize}
\end{corollary}

%
%
%

\subsection{Boolean Ultrapowers}

We will also need Boolean ultrapowers (starting in Section~\ref{FullBVModelsSec}). These are implicit in the work of Scott and Solovay \cite{Scott},  and made explicit by Vopenka \cite{Vopenka}. We follow the notation of Mansfield \cite{BooleanUltrapowers} and Hamkins and Sebald \cite{Hamkins}.

\begin{definition}\label{DefOfBoolUlt}
	Suppose $M \models T$. Let the set of all partitions of $\mathcal{B}$ by $M$, denoted $M^{\mathcal{B}}$, be the set of all functions $\mathbf{a}: M \to \mathcal{B}$, such that for all $a, b \in M$, $\mathbf{a}(a) \wedge \mathbf{a}(b) = 0$, and such that $\bigvee_{a \in M} \mathbf{a}(a) = 1$. Given $\phi(\mathbf{a}_i: i < n) \in \mathcal{L}(M^{\mathcal{B}})$, put $\|\phi(\mathbf{a}_i: i < n)\|_{\mathcal{B}} = \bigvee_{M \models \phi(a_i: i < n)}\bigwedge_{i < n} \mathbf{a}_i(a_i)$. (One must check that this does not change if we add dummy parameters to $\phi$, but this is straightforward.)
	
	Let $\mathbf{i}: M \to M^{\mathcal{B}}$ be the embedding sending $a \in M$ to the function $\mathbf{i}(a): M \to \mathcal{B}$ which takes the value $1$ on $a$, and $0$ elsewhere. We call this the pre-{\L}o{\'s} embedding.
\end{definition}

The following theorem is the conjunction of Corollary 1.2 and Theorem 1.4 of \cite{BooleanUltrapowers} (in the special case of a relational language).

\begin{theorem}\label{preLosThm}
	Suppose $M$ is a $\{0, 1\}$-valued structure and $\mathcal{B}$ is a complete Boolean algebra (so $M$ is also a full $\mathcal{B}$-valued structure). Then $M^{\mathcal{B}}$ is a full $\mathcal{B}$-valued $\mathcal{L}$-structure, and $\mathbf{i}: M \preceq M^{\mathcal{B}}$.
\end{theorem}

\section{Patterns}\label{SurveyLocalSec}

In this section,  we introduce the notion of patterns, which will give a method of computing $\lambda_{\hat{V}}(T)$ when $\hat{V}$ is $\aleph_1$-saturated.

We first recall a theorem due to Malliaris \cite{PhiTypes}, that says that ultrapowers are saturated if and only if they are locally saturated. We phrase it in the terminology of $\trianglelefteq^\times_{\aleph_1}$; Malliaris's proof translates to this context verbatim.

\begin{theorem}\label{KeislersOrderIsLocal}
	Suppose $V \models ZFC^-$ is transitive, $\mathbf{j}: V \preceq \hat{V}$ is $\aleph_1$-saturated, and $T \in V$ is a complete countable theory. Suppose $M \models T$ with $M \in V$, and $\lambda$ is a cardinal. Then the following are equivalent.
	\begin{itemize}
		\item[(A)] $\mathbf{j}_{\std}(M)$ is $\lambda^+$-pseudosaturated;
		\item[(B)] For every formula $\phi(x, \overline{y})$ and for every positive, pseudofinite $\phi$-type $p(x)$ over $M$ of cardinality at most $\lambda$, $p(x)$ is realized in $M$.
	\end{itemize}
\end{theorem}

The following is interchangeable with the terminology of characteristic sequences, arrays and diagrams of Malliaris \cite{CharSequence}, although we find patterns to be more convenient. 
\begin{definition}\label{PatternsDef}
	If $I$ is an index set, then a pattern on $I$ is some $\Delta \subseteq [I]^{<\aleph_0}$ which is closed under subsets. A $\Delta$-clique is a subset $X \subseteq I$ with $[X]^{<\aleph_0} \subseteq \Delta$. If $\Delta$ is a pattern on $I$ and $\Delta'$ is a pattern on $J$, then say that $\Delta'$ is an instance of $\Delta$ if for all $s \in [J]^{<\aleph_0}$ there is some map $f: s \to I$ such that for all $t \subseteq s$, $t \in \Delta'$ if and only if $f[t] \in \Delta$. Say that two patterns $\Delta, \Delta'$ are equivalent if they are instances of each other.

\end{definition}

Note that every pattern is equivalent to one on $\omega$.

\begin{definition}
Suppose $V \models ZFC^-$ is transitive, $\mathbf{j}: V \preceq \hat{V}$, and suppose $\Delta \in V$ is a pattern on $I$ (so $I \in V$). Then let $\lambda_{\hat{V}}(\Delta)$ be the least $\lambda$ such that there is some pseudofinite $X \subseteq \mathbf{j}(I)$ with $|X| \leq \lambda$, such that $[X]^{<\aleph_0} \subseteq \mathbf{j}(\Delta)$, and such that there is no $\hat{X} \in \mathbf{j}(\Delta)$ with $X \subseteq \hat{X}$. If there is no such $\lambda$ then let $\lambda_{\hat{V}}(\Delta) = \infty$.
\end{definition}

For any pattern $\Delta$ on $\omega$, and for any $\mathbf{j}: V \preceq \hat{V}$, we have that $\lambda_{\hat{V}}(\Delta) \geq \mathfrak{p}_{\hat{V}}$ by Theorem 5.3(A) of \cite{InterpOrdersUlrich}.

\begin{remark}
	Technically we should refer to the pair $(I, \Delta)$ in the definition of equivalence, but $I$ will always be clear from context.
	
	Also we should refer to $\lambda_{\hat{V}, \mathbf{j}}(\Delta)$ but $\mathbf{j}$ will be clear from context.
\end{remark}

\begin{lemma}\label{EquivLemma}
	Suppose $V \models ZFC^-$ is transitive, and $\mathbf{j}: V \preceq \hat{V}$. Suppose $\Delta$ is a pattern on $I$ and $\Delta'$ is a pattern on $I'$, with $\Delta, I, \Delta', I' \in V$. If $\Delta$ is an instance of $\Delta'$, then $\lambda_{\hat{V}}(\Delta) \geq \lambda_{\hat{V}}(\Delta')$. Hence, if $\Delta$ and $\Delta'$ are equivalent, then $\lambda_{\hat{V}}(\Delta) = \lambda_{\hat{V}}(\Delta')$.
\end{lemma}
\begin{proof}
	It suffices to prove the first part. Note that $V \models (\Delta$ is an instance of $\Delta'$), since this is a finitary condition.
	
	Suppose $X \subseteq \mathbf{j}(I)$ is pseudofinite with $|X| < \lambda_{\hat{V}}(\Delta')$; choose $X \subseteq \hat{X}$ with $\hat{X} \in \hat{V}$ finite in the sense of $\hat{V}$; by replacing $\hat{X}$ with $\hat{X} \cap \mathbf{j}(I)$ we can suppose $\hat{X} \subseteq \mathbf{j}(I)$. Since $\mathbf{j}(\Delta)$ is an instance of $\mathbf{j}(\Delta')$ in $\hat{V}$, we can find some $\hat{X}' \subseteq \mathbf{j}(I')$ finite in the sense of $\hat{V}$ and some map $\hat{f}: \hat{X} \to \hat{X}'$ in $\hat{V}$, such that for all $\hat{s} \subseteq \hat{x}$ in $\hat{V}$, $\hat{s} \in \mathbf{j}(\Delta)$ if and only if $\hat{f}[\mathbf{s}] \in \mathbf{j}(\Delta')$. Define $X' = \hat{f}[X]$ (as computed in $\mathbb{V}$). Then by definition of $\lambda_{\hat{V}}(\Delta')$ we can find some $\hat{s}' \in \mathbf{j}(\Delta')$ with $X' \subseteq \hat{s}'$. Let $\hat{s} = \hat{f}^{-1}(\hat{s}')$; then $\hat{s} \in \mathbf{j}(\Delta)$ and $X \subseteq \hat{s}$, as desired.
\end{proof}

We now connect this with pseudosaturation:

\begin{definition}
	
	Suppose $T$ is a complete countable theory, $\phi(\overline{x}, \overline{y})$ is a formula of $T$ and $\Delta$ is a pattern on $I$. Then say that $\phi(\overline{x}, \overline{y})$ admits $\Delta$ if we can choose $(\overline{a}_i: i \in I)$ from $\mathfrak{C}^{|\overline{y}|}$ (where $\mathfrak{C}$ is the monster model of $T$, or just any $|I|^+$-universal model), such that for every $s \in [I]^{<\aleph_0}$, $\exists \overline{x} \bigwedge_{i \in s} \phi(\overline{x}, \overline{a}_i)$ is consistent if and only if $s \in \Delta$ (whether or not $\phi$ admits $\Delta$ depends on $T$; if there is confusion, we will say that $(T, \phi)$ admits $\Delta$). Say that $T$ admits $\Delta$ if some formula of $T$ does.

	Suppose $T$ is a complete countable theory, $M \models T$ and $\phi(\overline{x}, \overline{y})$ is a formula of $T$. Then we can form a pattern $\Delta_{M, \phi} := \{s \in [M^{|\overline{y}|}]^{<\aleph_0}: M \models \exists \overline{x} \bigwedge_{\overline{a} \in s} \phi(\overline{x}, \overline{a})\}$. Then for all patterns $\Delta$, $\phi$ admits $\Delta$ if and only if $\Delta$ is an instance of $\Delta_{M, \phi}$. Hence for all $M, N \models T$, $\Delta_{M, \phi}$ and $\Delta_{N, \phi}$ are equivalent. Write $\Delta_\phi = \Delta_{M, \phi}$, for some arbitrary choice of $M \models T$. We will only refer to $\Delta_\phi$ in contexts where we just need its equivalence class.

If $T$ is a complete countable theory in $V$ and $\phi(\overline{x}, \overline{y})$ is a formula of $T$, then let $\lambda_{\hat{V}}(T, \phi) = \lambda_{\hat{V}}(\Delta_\phi)$. By Lemma~\ref{EquivLemma} this only depends on the equivalence class of $\Delta_\phi$, and so is well-defined.

For each $n$, let $\lambda_{\hat{V}}^{loc, n}(T)$ be the minimum over all formulas $\phi(\overline{x}, \overline{y})$ of $T$ with $|\overline{x}| \leq n$ of $\lambda_{\hat{V}}(T, \phi)$; so this is a descending sequence of cardinals (which necessarily stabilizes). Let $\lambda_{\hat{V}}^{loc}(T) = \min_n \lambda_{\hat{V}}^{loc, n}(T)$.
\end{definition}

\begin{theorem}\label{CardInvarTheoremInterp}
	Suppose $V \models ZFC^-$ is transitive and $\mathbf{j}: V \preceq \hat{V}$ is $\omega$-nonstandard and $T \in V$ is a complete countable theory. Then $\lambda_{\hat{V}}(T) \leq \lambda_{\hat{V}}^{loc}(T,\phi)$. If $\hat{V}$ is $\aleph_1$-saturated, then $\lambda_{\hat{V}}(T) = \lambda_{\hat{V}}^{loc}(T) = \lambda_{\hat{V}}^{loc, 1}(T)$.
\end{theorem}
\begin{proof}
	This is basically Theorem~\ref{KeislersOrderIsLocal}, restated. The point is the following: suppose $M \models T$ with $M \in V$, and let $\phi(\overline{x}, \overline{y})$ be a formula of $T$. Let $\Delta = \Delta_{M, \phi} = \{s \in [M^{|\overline{y}|}]^{<\aleph_0}: M \models \exists \overline{x} \bigwedge_{\overline{b} \in s} \phi(\overline{x}, \overline{b})\}$ and let $\hat{\Delta} = \mathbf{j}(\Delta)$. Then positive $\phi$-types $p(\overline{x})$ over $\mathbf{j}_{\std}(M)$ correspond exactly to subsets $X$ of $\mathbf{j}_{\std}(M)^{|\overline{y}|}$ with $[X]^{<\aleph_0} \subseteq \hat{\Delta}$, and, assuming $p(\overline{x})$ is pseudofinite, $p(\overline{x})$ is realized in $\mathbf{j}_{\std}(M)$ if and only if there is some $\hat{X} \in \hat{\Delta}$ with $X \subseteq \hat{X}$. Moreover, it suffices to consider types in a single variable $x$.
\end{proof}

Malliaris proved the following corollary as Lemma 5.14 of \cite{MalliarisFlex} for Keisler's order $\trianglelefteq$, under the terminology of characteristic sequences:

\begin{corollary}
Suppose $T_0, T_1$ are complete countable theories. Suppose that for every pattern $\Delta$, if $T_0$ admits $\Delta$ then $T_1$ admits $\Delta$. Then $T_0 \trianglelefteq^\times_{\aleph_1} T_1$, and hence $T_0 \trianglelefteq^*_{\aleph_1} T_1$ and $T_0 \trianglelefteq T_1$.
\end{corollary}

Although this is not the line of investigation of the current work, this theorem suggests a natural ordering on theories (first proposed by Shelah \cite{SH702} under notation similar to Malliaris's), namely: put $T_0 \leq T_1$ if and only if for every pattern $\Delta$, if $T_0$ admits $\Delta$ then so does $T_1$ (it is enough to consider just formulas $\phi(x, \overline{y})$ where $x$ is a single variable; this may be preferable). $\leq$ detects many more properties than do the interpretability orders; for instance, it follows from our work that if $T_0$ has $IP$ and $T_1$ is $NIP$, then $T_0 \not \leq T_1$, and so $DLO$ is not maximal in $\leq$. Shelah defines the ``straightly maximal" theories to be the maximal class of $\leq$; one simple example is $\mbox{Th}(\mathcal{P}(\omega), \subseteq)$.

\section{A Minimal Unstable Theory}\label{RandGraphSec}

Let $T_{rg}$ be the theory of the random graph. Malliaris proved in \cite{HypSeq} that $T_{rg}$ is a $\trianglelefteq$-minimal unstable theory. Malliaris and Shelah prove in \cite{InterpOrders} that $T_{rg}$ is a $\trianglelefteq^{*}_1$-minimal unstable theory, although that proof has some additional complications. In fact, Malliaris's proof goes through  to show that $T_{rg}$ is a $\trianglelefteq^{\times}_1$-minimal theory; the argument is even simplified.

The following essentially describes an $(\omega, 2)$-array as in \cite{CharSequence}.

\begin{definition}
	Let $\Delta({{IP}})$ be the pattern on $\omega \times 2$, defined to be the set of all $s \in [\omega \times 2]^{<\aleph_0}$ such that for all $n < \omega$, $\{(n, 0), (n, 1)\} \not \subseteq s$, i.e. $s$ is a partial function from $\omega$ to $2$. (Think of $(n, 0)$ as being $\lnot (n, 1)$.)
\end{definition}

And the following is Claim 3.7 from \cite{CharSequence}.

\begin{lemma}\label{IPLemma0}
	Suppose $T$ is a complete countable theory and $\phi(\overline{x}, \overline{y})$ is a formula of $T$. Define $\theta(\overline{x}, \overline{y}_0, \overline{y}_1) = \phi(\overline{x}, \overline{y}_0) \land \lnot \phi(\overline{x}, \overline{y}_1)$. Then $\theta(\overline{x}, \overline{y}_0 \overline{y}_1)$ admits $\Delta({IP})$ if and only if $\theta(\overline{x}, \overline{y}_0 \overline{y}_1)$ has the independence property, which is the case if and only if $\phi(\overline{x}, \overline{y})$ has the independence property.
	
	Thus $T$ has the independence property if and only if it admits $\Delta({IP})$.
\end{lemma}
%
%
%

The following lemma is helpful in understanding the invariant $\lambda_{\hat{V}}(\Delta(IP))$. It is a translation of remarks in \cite{HypSeq} into the context of models of $ZFC^-$, for instance see the discussion in Example 2 in Section 3.2.

\begin{lemma}\label{IPLemma1}
	Suppose $V \models ZFC^-$ is transitive, and $\mathbf{j}: V \preceq \hat{V}$. Then $\lambda_{\hat{V}}(\Delta(IP))$ is the least $\lambda$ such that for some $\hat{n} < \hat{\omega}$, there are disjoint, pseudofinite $X_0, X_1 \subseteq \hat{V}$ each of size at most $\lambda$, such that there are no disjoint $\hat{X}_0, \hat{X}_1 \in \hat{V}$ with each $X_i \subseteq \hat{X}_i$.
\end{lemma}
\begin{proof}
	Let $\lambda$ be the least cardinal as in the statement of the lemma (possibly $\infty$).
	
	We first show that $\lambda_{\hat{V}}(\Delta(IP)) \leq \lambda$. So suppose $X \subseteq \hat{n} \times 2$ is given, with $|X| < \lambda$ and with $[X]^{<\aleph_0} \subseteq \mathbf{j}(\Delta({IP}))$; i.e. $X$ is a partial function from $\hat{n}$ to $2$. For each $i < 2$, define $X_i = \{\hat{m} < \hat{n}: (\hat{m}, i) \in X\}$; by hypothesis, there exist disjoint $\hat{X}_i: i < 2$ in $\hat{V}$  with each $X_i \subseteq \hat{X}_i$; by replacing $\hat{X}_i$ by $\hat{X}_i \cap \hat{n}$, we can suppose $\hat{X}_i \subseteq \hat{n}$. Then $X \subseteq \hat{X}_0 \times \{0\} \cup \hat{X}_1 \times \{1\} \in \mathbf{j}(\Delta(IP))$. 
	
	Conversely, we show that $\lambda \leq \lambda_{\hat{V}}(\Delta(IP))$. So suppose we are given disjoint, pseudofinite $X_0, X_1 \subseteq \hat{V}$ each of cardinality less than $\lambda$. Choose $\hat{Y} \in \hat{V}$, finite in the sense of $\hat{V}$, with $X_0, X_1 \subseteq \hat{Y}$, and choose a bijection $\hat{f}: \hat{Y} \to |\hat{Y}|$ in $\hat{V}$. In $\mathbb{V}$, define $X = \hat{f}[X_0] \times \{0\} \cup \hat{f}[X_1] \times \{1\}$. Note that $|X| \leq \lambda$ and $[X]^{<\aleph_0} \subseteq \mathbf{j}(\Delta(IP))$ and $X$ is pseudofinite. By hypothesis, there existss $\hat{X} \in \mathbf{j}(\Delta(IP))$ with $X \subseteq \hat{X}$. Let $\hat{X}_i = \hat{f}^{-1}[\{\hat{n}: (\hat{n}, i) \in \hat{X}\}]$ for each $i < 2$, then $\hat{X}_0, \hat{X}_1 \in \hat{V}$ are disjoint and each $X_i \subseteq \hat{X}_i$.
\end{proof}

Thus:
\begin{theorem}\label{RandGraphInterp}
	Suppose $V \models ZFC^-$ is transitive, and $\mathbf{j}: V \preceq \hat{V}$ is $\omega$-nonstandard, and $\lambda$ is an infinite cardinal. Then the following are equivalent:
	
	\begin{itemize}
		\item[(A)] $\hat{V}$ $\lambda^+$-pseudosaturates $T_{rg}$;
		\item[(B)] $\hat{V}$ $\lambda^+$-pseudosaturates some unstable theory;
		\item[(C)] $\lambda < \lambda_{\hat{V}}(\Delta(IP))$.
	\end{itemize}
\end{theorem}
\begin{proof}
	(A) implies (B) is trivial.
	
	(B) implies (C): suppose $T \in V$ is unstable such that $\hat{V}$ $\lambda^+$-pseudosaturates $T$, i.e. $\lambda < \lambda_{\hat{V}}(T)$. Now $T$ either has $SOP_2$ or else $IP$, by Theorem~\ref{NonStableDichotomy}. If $T$ has $SOP_2$ then $\lambda_{\hat{V}}(T) \leq \mathfrak{p}_{\hat{V}} \leq \lambda_{\hat{V}}(\Delta(IP))$. If on the other hand $T$ has $IP$, then $T$ admits $\Delta({IP})$ so we get $\lambda_{\hat{V}}(T) \leq \lambda_{\hat{V}}(\Delta(IP))$ in any case.
	
	(C) implies (A): let $M \models T_{rg}$ with $M \in V$, and let $p(x)$ be a pseudofinite partial type over $\mathbf{j}_{\std}(M)$ of cardinality at most $\lambda$; say $p(x)$ is a complete type over $A \subseteq \mathbf{j}_{\std}(M))$, where $A$ is pseudofinite. Let $X_0 = \{a \in A: R(x, a) \in p(x)\}$ (defined in $\mathbb{V}$) and let $X_1 = \{a \in A: \lnot R(x, a) \in p(x)\}$, and conclude by Lemma~\ref{IPLemma1}.
\end{proof}

We immediately get the following corollary.

\begin{corollary}\label{RandGraphMin}
	$T_{rg}$ is a $\trianglelefteq^\times_{1}$-minimal unstable theory. That is, if $T$ is unstable then $T_{rg} \trianglelefteq^\times_{1} T$. Thus, this also holds for $\trianglelefteq^\times_\kappa$ and $\trianglelefteq^*_\kappa$ for every $\kappa$, and also for $\trianglelefteq$.
\end{corollary}

\section{A Minimal Unsimple Theory}\label{StarRandGraphSec}

In \cite{HypSeq}, Malliaris proved the existence of a $\trianglelefteq$-minimal $TP_2$ theory (namely, $T_{feq}$). In view of Theorem~\ref{NonSimpleDichotomy} and Theorem~\ref{SOP2max} (due to Malliaris and Shelah \cite{pEqualsTref} for Keisler's order), $T_{feq}$ must also be a $\trianglelefteq$-minimal unsimple theory. More recently in \cite{InterpOrders}, Malliaris and Shelah prove that $T_{feq}$ is a $\trianglelefteq^*_1$-minimal unsimple theory.

 We perform the routine translations into the language of $\trianglelefteq^\times_1$. However, we prefer to use the following as our flagship unsimple theory:

\begin{definition}
	Let $T_{rf}$ be the theory of the random binary function. That is, $T_{rf}$ is the model completion of the empty theory in the language containing a single binary function symbol $F$.
\end{definition}

$T_{rf}$ is shown to be $NSOP_1$ and to admit quantifier elimination in \cite{SkolemNSOP1}. In particular, $T_{rf}$ is $NSOP_2$. Further, $T_{rf}$ is $TP_2$ via the formula $f(x, y_0) = y_1$.

We now proceed as in the previous section. The following definition is equivalent to the notion of $(\omega, \omega, 1)$-arrays from \cite{CharSequence}.

\begin{definition}
	Let $\Delta({TP}_2)$ be the pattern on $\omega \times \omega$ consisting of all $s \in [\omega \times \omega]^{<\aleph_0}$ satisfying: for all $n < \omega$, $|s \cap \{n\} \times \omega| \leq 1$, i.e. $s$ is a partial function from $\omega$ to $\omega$.
\end{definition}

The following is then trivial. (This is Claim 3.8 of \cite{CharSequence}, although our choice of definition of $TP_2$ absorbs all of the work.)

\begin{lemma} \label{TP2Lemma0}
	Suppose $\phi(\overline{x}, \overline{y})$ is a formula of $T$. Then $\phi(\overline{x}, \overline{y})$ has $TP_2$ if and only if $\phi(\overline{x}, \overline{y})$ admits $\Delta({TP}_2)$. Thus $T$ has $TP_2$ if and only if $T$ admits $\Delta({TP}_2)$.
\end{lemma}

The following is essentially Theorem 6.9 of \cite{HypSeq}. To explain the terminology: we are viewing a partial function $f: \hat{V} \to \hat{V}$ as a subset of $\hat{V}^2 \subseteq \hat{V}$, so the statement that $f$ is pseudofinite just means that the domain and range of $f$ are pseudofinite. 

\begin{lemma}\label{TP2Lemma1}
	Suppose $V \models ZFC^-$ is transitive, and $\mathbf{j}: V \preceq \hat{V}$ is $\omega$-nonstandard, and $\lambda$ is a cardinal. Then $\lambda_{\hat{V}}(\Delta(TP_2))$ is the least $\lambda$ such that there is a pseudofinite partial function $f: \hat{V} \to \hat{V}$ which cannot be extended to an internal partial function $\hat{f}$ from $\hat{V}$ to $\hat{V}$.
\end{lemma}
\begin{proof}
	Let $\lambda$ be the least such cardinal as in the statement of the lemma.
	
	First we show that $\lambda_{\hat{V}}(\Delta(TP_2)) \leq \lambda$. Given some pseudofinite $f \subseteq \hat{\omega} \times \hat{\omega}$ with $[f]^{<\aleph_0} \subseteq \mathbf{j}(\Delta(TP_2))$, and of cardinality less than $\lambda$, note that $f$ is a partial function from $\hat{V}$ to $\hat{V}$, and so we can find some internal partial function $\hat{f}$ from $\hat{V}$ to $\hat{V}$ extending $f$. Choose $\hat{n}< \hat{\omega}$ large enough so that $f \subseteq \hat{n} \times \hat{n}$; let $\hat{X} = \{\hat{m} < \hat{n}: \hat{f}(\hat{m}) \mbox{ is defined and } < \hat{n}\}$. Then $\hat{f} \restriction_{\hat{X}} \in \mathbf{j}(\Delta(TP_2))$.
	
	Conversely, we show that $\lambda \leq \lambda_{\hat{V}}(\Delta(TP_2))$. Suppose $f$ is a pseudofinite partial function from $\hat{V}$ to $\hat{V}$ of cardinality less than $\lambda_{\hat{V}}(\Delta(TP_2))$. We can find some $\hat{X} \in \hat{V}$, finite in the sense of $\hat{V}$, such that $f$ is a partial function from $\hat{X}$ to $\hat{X}$. By relabeling, we can suppose $\hat{X} = \hat{n} < \hat{\omega}$. Then $[f]^{<\aleph_0} \subseteq \mathbf{j}(\Delta(TP_2))$, and thus we can find $\hat{f} \in \mathbf{j}(\Delta(TP_2))$ with $f \subseteq \hat{f}$. Then $\hat{f}$ is as desired.
\end{proof}

Putting it all together:

\begin{theorem}\label{StarRandGraphInterp}
	Suppose $V \models ZFC^-$ is transitive, and $\mathbf{j}: V \preceq \hat{V}$ is $\omega$-nonstandard, and $\lambda$ is a cardinal. Then the following are equivalent:
	
	\begin{itemize}
		\item[(A)] $\hat{V}$ $\lambda^+$-pseudosaturates $T_{rf}$;
		\item[(B)] $\hat{V}$ $\lambda^+$-pseudosaturates some unsimple theory;
		\item[(C)] $\lambda < \lambda_{\hat{V}}(\Delta(TP_2))$.
	\end{itemize}
\end{theorem}
\begin{proof}
	(A) implies (B) is trivial.
	
	(B) implies (C): suppose $T \in V$ is unsimple and $\hat{V}$ $\lambda^+$-pseudosaturates $T$, i.e. $\lambda < \lambda_{\hat{V}}(T)$. By Theorem~\ref{NonSimpleDichotomy}, $T$ either has $SOP_2$ or else $TP_2$; if $T$ has $SOP_2$ then $\lambda_{\hat{V}}(T) = \mathfrak{p}_{\hat{V}} \leq \lambda_{\hat{V}}(\Delta(TP_2))$. If on the other hand $T$ has $TP_2$, then $T$ admits $\Delta(TP_2)$ so we get $\lambda_{\hat{V}}(T) \leq \lambda_{\hat{V}}(\Delta(TP_2))$ in any case.
	
	(C) implies (A): Let $F: \omega^2 \to \omega$ be such that $(\omega, F) \models T_{rf}$, and write $\hat{F}= \mathbf{j}(F)$ (we also use $F$ to denote the symbol in the language). Let $p(x)$ be a pseudofinite partial type over $(\hat{\omega}, \hat{F})$, say $p(x)$ is over $X \subseteq \hat{n}$ with $|X| \leq \lambda$. We need to show $p(x)$ is realized in $(\hat{\omega}, \hat{F})$. We can suppose $p(x)$ is nonalgebraic. 
	
	Let $\kappa = |\hat{n}|$ as computed in $\mathbb{V}$; i.e. $\kappa$ is the cardinality of $\{\hat{m}  \in \hat{V}: \hat{m}< \hat{n}\}$ in $\mathbb{V}$. I claim that $\kappa> \lambda$. Suppose towards a contradiction that $\kappa \leq \lambda$. In $\mathbb{V}$, choose a bijection $f: (\hat{n}-1)\to \hat{n}$. By Lemma~\ref{TP2Lemma1}, we can find some internal partial function $\hat{f}$ extending $f$. But then $\hat{f} \restriction_{\hat{n}-1} = f$, and $f$ cannot be internal, contradiction.
	
	Thus we can find some $Y \subseteq \hat{n}$ such that $X \subseteq Y$ and $|Y \backslash X| = \lambda$. Extend $p(x)$ to a complete type $q(x)$ over $Y$ such that $q(x)$ is induced by some function $f: (Y \cup \{x\})^2 \to Y \cup \{x\}$. More precisely, $q(x)$ is nonalgebraic, and for every $a,b \in Y \cup \{x\}$, if we write $c = f(a, b)$, then $q(x) \models F(a, b) = c$. Note that $q(x)$ is logically implied by $\{``F(a, b) = c": a, b \in Y \cup \{x\}, c = f(a, b)\}$. 
	
	Since $\lambda < \lambda_{\hat{V}}(\Delta(TP_2))$, we can find some function $\hat{f}: (\hat{n} \cup \{x\})^2 \to \hat{n} \cup \{x\}$ extending $f$. Thus we can find $a_* \geq \hat{n}$ such that $\hat{F} \restriction_{\hat{n} \cup \{a_*\}}$ is given by $\hat{f}$; then $a_*$ realizes $q(x)$.
\end{proof}

We immediately get the following corollary.

\begin{corollary}\label{StarRandGraphMin}
	$T_{rf}$ is a $\trianglelefteq^\times_{1}$-minimal unsimple theory. That is, if $T$ is unsimple then $T_{rf} \trianglelefteq^\times_{1} T$. Thus, this also holds for $\trianglelefteq^\times_\kappa$ and $\trianglelefteq^*_\kappa$ for every $\kappa$, and also for $\trianglelefteq$.
\end{corollary}

\section{A Minimal Nonlow Theory}\label{TCasSec}

In this section, we proceed similarly to Sections~\ref{RandGraphSec} and ~\ref{StarRandGraphSec} to show that there is a minimal nonlow theory in $\trianglelefteq^\times_1$, namely $T_{nlow}$. In \cite{LowDividingLine} I proved that a different theory, $T_{Cas}$ (introduced by Casanovas in \cite{Casanovas}) is a minimal nonlow theory in Keisler's order. The arguments there show that $T_{Cas}$ is in fact a $\trianglelefteq^\times_{\aleph_1}$-minimal nonlow theory; however, Malliaris and Shelah show in \cite{InterpNew} that supersimplicity is a dividing line in $\trianglelefteq^*_1$, and hence $T_{Cas}$ is not a $\trianglelefteq^*_1$-minimal nonlow theory (seeing as it is not supersimple). $T_{nlow}$ is the first example of a supersimple nonlow theory \cite{SupersimpleNonlow}, due to Casanovas and Kim.

We now describe $T_{nlow}$. The language $\mathcal{L}_{nlow}$ is $(R, E, P, Q, U_n, P_n, Q_n^i, F_n: 1 \leq n < \omega, i < 2)$, where $P, Q, U_n, P_n, Q_n^i$ are each unary relation symbols, and $R, E$ are binary relation symbols, and $F$ is an $n$-ary relation symbol. (Our notation differs from \cite{SupersimpleNonlow}: our $P$ is their $Q^0$, our $Q$ is their $Q^1$, our $U_n$ is their $P_n$, our $P_n$ is their $Q_n^0$, and our $Q_n^i$ is their $Q_n^{i+1}$.)

 $T_{nlow}$ is the model completion of the following axioms:

\begin{enumerate}
	\item The universe is the disjoint union of $P$ and $Q$, both infinite;
	\item $E$ is an equivalence relation on the universe;
	\item $R \subseteq P \times Q$ and $R \subseteq E$;
	\item Each $U_n$ is an equivalence class of $E$; 
	\item Each $P_n = U_n \cap P$;
	\item $U_n \cap Q$ is the disjoint union of $Q_n^0$ and $Q_n^1$;
	\item For each $a \in P_n$, the set $u_a := \{b \in Q_n^0: R(a, b)\}$ has exactly $n$ elements; moreover $a \mapsto u_a$ is a bijection from $P_n$ to $[Q_n^0]^n$;
	\item $F_n$ induces the bijection (also denoted $F_n$) from $[Q_n^0]^n$ to $P_n$ given by $F_n(u_a) = a$. More formally, for all $(b_0, \ldots, b_{n-1})$, if some $b_i \not \in Q_n^0$ or else there are $i < j$ with $b_i = b_j$, then $F_n(b_0, \ldots, b_{n-1}) = b_0$. Otherwise, $F_n(b_0, \ldots, b_{n-1}) \in P_n$, and $u_{F_n(b_0, \ldots, b_{n-1})} = \{b_0, \ldots, b_{n-1}\}$.
\end{enumerate}

As notation, we will let $U_\omega$ be the complement of $\bigcup_{1 \leq n < \omega} U_n$, and we will let $P_\omega = P \backslash \bigcup_{1 \leq n < \omega} P_n$, and we will let $Q_\omega = Q \backslash \bigcup_{1 \leq n < \omega, i < 2} Q_n^i$. These are type-definable sets, and so we can view them also as partial types.

In \cite{SupersimpleNonlow} it is shown that $T_{nlow}$ is well-defined. Moreover, $X$ is algebraically closed if and only if $X$ is closed under each $F_n$, and for all $a \in P_n$, $u_a \subseteq X$. Further, $X$ is supersimple, with forking relation given by: $A \forkindep_C B$ if and only if $\acl(AC) \cap \acl(BC) \subseteq \acl(C)$, and further, if $a \in A \backslash \bigcup_n P_n$ and $a$ is not $E$-related to any element of $C$, then $a$ is not $E$-related to any element of $B$. (This is equivalent to $\acl(AC) \cap \acl(BC) \subseteq \acl(C)$ once we eliminate the imaginaries $[a/E]$.) Finally, the formula $R(x, y)$ visibly witnesses that $T_{nlow}$ has the finite dividing property, and hence is nonlow.

$T_{nlow}$ does not have quantifier elimination (note that the axioms above are not universal). However, whenever $A$ is algebraically closed, then complete types over $A$ are determined by their quantifier-free part. The following lemma follows from the proof of Proposition 4.2 in \cite{SupersimpleNonlow}:

\begin{lemma}\label{TCasLemma1}
	Let $M \models T_{nlow}$ and let $C \subseteq M$ be algebraically closed. Write $C = A \cup B$ where $A = C \cap P^M$ and $B = C \cap Q^M$. Write $A_n = A \cap P_n$ for each $1 \leq n < \omega$, and write $B_n^i = B \cap Q_n^i$ for each $1 \leq n < \omega$ and $i < 2$.
	\begin{itemize} 
		\item[(I)] Suppose $1 \leq n < \omega$. Let $X \subseteq [B_{n}^0]^{n-1} \times B_{n}^1$ be given. Let $p^{I}_X(x)$ be the partial type over $C$ which asserts:
		\begin{itemize}
			\item $Q_{n}^0(x)$;
			\item $\lnot R(a, x)$ for each $a \in A_{n}$
			\item For every $v \in [B_{n}^0]^{n-1}$ and for every $b \in B_{n}^1$, $R(F_n(v \cup \{x\}), b)$ if and only if $(v, b) \in X$.
		\end{itemize}
	
	 Then $p^I_X(x)$ generates a complete type over $C$ that does not fork over $\emptyset$. Moreover, all nonalgebraic complete types over $C$ extending $Q_n^0(x)$ are of this form. 
	
	\item[(II)] Suppose $1 \leq n < \omega$. let $X \subseteq A$ be given. Let $p^{II}_X(x)$ be the partial type over $C$ that asserts:
	
	\begin{itemize}
		\item $Q_n^1(x)$;
		\item $x \not= b$ for each $b \in B_{n}^1$;
		\item For each $a \in A_n$, $R(a, x)$ holds if and only if $a \in X$.
	\end{itemize}
	
	Then $p_X^{II}(x)$ generates a complete type over $C$ that does not fork over $\emptyset$. Moreover, all nonalgebraic complete types over $C$ extending $\{Q(x)\} \cup Q_n^1(x)$ are of this form. 
	
	\item[(III)] Suppose $c \in  C \cap U_\omega$ and $X \subseteq A \cap [c]_E$. Then let $p^{III}_{X, c}(x)$ be the partial type over $C$ which asserts:
		\begin{itemize}
			\item $x \not= b$ for each $b \in B \cap [c]_E$;
			\item $x E c$;
			\item $Q(x)$;
			\item For every $a \in A \cap [c]_E$, $R(a, x)$ holds if and only if $a \in X$.
		\end{itemize}
	
	Then $p^{III}_{X, c}(x)$ generates a complete type over $C$ that does not fork over $\emptyset$. Moreover, all nonalgebraic complete types over $C$ extending $Q_\omega(x) \land  x E c$ for some $c \in C$ are of this form.
	
	\item[(IV)] Suppose $c \in  C \cap U_\omega$ and $X \subseteq B \cap [c]_E$. Then let $p^{IV}_{X, c}(x)$ be the partial type over $C$ which asserts:
	\begin{itemize}
		\item $x \not= a$ for each $a \in A \cap [c]_E$;
		\item $x E c$;
		\item $P(x)$;
		\item For every $b \in B \cap [c]_E$, $R(x, b)$ holds if and only if $b \in X$.
	\end{itemize}
	
	Then $p^{IV}_{X, c}(x)$ generates a complete type over $C$ that does not fork over $\emptyset$. Moreover, all nonalgebraic complete types over $C$ extending $P_\omega(x) \land  x E c$ for some $c \in C$ are of this form.
	
	\item[(V)] There are unique types over $C$ extending $P_\omega(x)$ and $Q_\omega(x)$ respectively, which additionally assert that $x$ is not $E$-related to any element of $C$. Further, for each $n < \omega$ and for each nonalgebraic type $p(x)$ over $C$ extending $P_n(x)$, there is a type $q(\overline{x}) \in S^n(C)$ extending $\bigwedge_{i < n} Q^0_n(x_i)$, such that $p(x)$ is realized in $M$ if and only if $q(\overline{x})$ is realized in $M$; further $q(\overline{x})$ can be chosen independently of the choice of $M \supseteq C$. Namely, for some or any realization of $p(x)$ in the monster model, let $\overline{b}$ be an enumeration of $u_a$ and let $q(\overline{x}) = tp(\overline{b})$; this works because $a$ and $\overline{b}$ are interdefinable.
	\end{itemize}
\end{lemma}

We now introduce the relevant patterns.

\begin{definition}
	Given $I \subseteq \omega \backslash \{0\}$ infinite, let $\Delta_I(FDP)$ be the pattern on $I \times \omega$ defined by: $s \in \Delta_I(FDP)$ if $s \in \{m\} \times [\omega]^{\leq m}$ for some $m \in I$. Let $\Delta_I^*(FDP)$ be the pattern on $I \times \omega$, defined to be all $s$ with each $|s \cap \{m \} \times \omega| \leq m$.
	
	Write $\Delta(FDP) = \Delta_{\omega \backslash \{0\}}(FDP)$.
\end{definition}

Easily, $\lambda_{\hat{V}}(\Delta(TP_2)) \leq \lambda_{\hat{V}}(\Delta(FDP))$. Also, $T_{nlow}$ admits $\Delta(FDP)$. 

The following is straightforward.

\begin{lemma}\label{nlowLemma0}
	Suppose $\phi(\overline{x}, \overline{y})$ is a formula of $T$. Then $\phi(\overline{x}, \overline{y})$ has the finite dividing property if and only if for some infinite $I \subseteq \omega \backslash \{0\}$, $\phi(\overline{x}, \overline{y})$ admits some $\Delta$ with $\Delta_I(FDP) \subseteq \Delta \subseteq \Delta_I^*(FDP)$. Hence $T$ has the finite dividing property if and only if $T$ admits some such $\Delta$.
\end{lemma}
\begin{proof}
	Suppose $\phi(\overline{x}, \overline{y})$ admits some such $\Delta$, via $(\overline{a}_{m, n}: (m, n) \in I \times \omega)$. Then by compactness and Ramsey's theorem, for each $m \in I$ we get an indiscernible sequence $(\overline{b}_{m, n}: n < \omega)$ such that $\{\phi(\overline{x}, \overline{b}_{m, n}): n < \omega\}$ is $m$-consistent but $m+1$-inconsistent. Hence $\phi(\overline{x}, \overline{y})$ has the finite dividing property.
	
	Conversely, suppose $\phi(\overline{x}, \overline{y})$ has the finite dividing property; choose $I \subseteq \omega \backslash \{0\}$ infinite, and indiscernible sequences $((\overline{b}^m_n: n < \omega): m \in I)$, so that each $\{\phi(\overline{x}, \overline{b}^m_n): m \in I\}$ is $m$-consistent but $m+1$-inconsistent. Let $\Delta$ be the set of all $s \in [\omega \times \omega]^{<\aleph_0}$ such that $\{\phi(\overline{x}, \overline{b}^m_n): (m, n) \in s\}$ is consistent. Clearly, $\Delta_I(FDP) \subseteq \Delta \subseteq \Delta_I^*(FDP)$.
\end{proof}

The following allows us to compare the various $\lambda_{\hat{V}}(\Delta)$'s from Lemma~\ref{nlowLemma0}; in particular $\lambda_{\hat{V}}(\Delta(FDP))$ is maximal among them. As convenient notation, if $\hat{V} \models ZFC^-$ and $\hat{n} < \hat{\omega}$, then let $[\hat{V}]^{\leq \hat{n}}$ denote $\{\hat{u} \in \hat{V}: \hat{V} \models |\hat{u}| \leq \hat{n}\}$. 

\begin{lemma}\label{nlowLemma1}
	Suppose $V$ is a transitive model of $ZFC^-$, and $\mathbf{j}: V \preceq \hat{V}$ is $\omega$-nonstandard. Suppose $I \subseteq \omega \backslash \{0\}$ is infinite, and $\Delta \in V$ is such that $\Delta_I(FDP) \subseteq \Delta \subseteq \Delta_I^*(FDP)$. If $\lambda < \lambda_{\hat{V}}(\Delta)$, then for every $\hat{m}_* < \hat{\omega}$ with $\hat{m}_*$ nonstandard, and for every pseudofinite $X \subseteq \hat{V}$ of cardinality at most $\lambda$, there is $\hat{X} \in [\hat{V}]^{\leq \hat{m}_*}$ with $X \subseteq \hat{X}$. If $\Delta = \Delta_I(FDP)$, then the converse is true as well.
	
	Hence $\lambda_{\hat{V}}(\Delta) \leq \lambda_{\hat{V}}(\Delta_I(FDP)) = \lambda_{\hat{V}}(\Delta(FDP))$.
\end{lemma}
\begin{proof}
	Suppose first $\lambda < \lambda_{\hat{V}}(\Delta)$, and $\hat{m}_*$, $X$ are as above. By relabeling, we can suppose $X \subseteq \hat{n}_*$ for some $\hat{n}_* < \hat{\omega}$. By decreasing $\hat{m}_*$, we can suppose $\hat{m}_* \in \mathbf{j}(I)$ while keeping $\hat{m}_*$ nonstandard. Write $Y = \{(\hat{m}_*, \hat{n}): \hat{n} \in X\}$. Since $\hat{m}_*$ is nonstandard we have $[Y]^{<\aleph_0} \subseteq \mathbf{j}(\Delta)$, thus we can find $\hat{Y} \in \mathbf{j}(\Delta)$ with $Y \subseteq \hat{Y}$. Let $\hat{X} = \{\hat{n} < \hat{n}_*: (\hat{m}_*, \hat{n}) \in \hat{Y}\}$; then $\hat{X} \in [\hat{n}_*]^{\leq \hat{m}_*}$ with $X \subseteq \hat{X}$.
	
	Next, suppose $\Delta = \Delta_I(FDP)$; let $Y \subseteq \hat{n}_* \times \hat{n}_*$ be of size at most $\lambda$ with $[Y]^{<\aleph_0} \subseteq \mathbf{j}(\Delta_I(FDP))$. Then $Y \subseteq  \{\hat{m}_*\} \times \hat{n}_*$ for some $\hat{m}_* < \hat{n}_*$ with $\hat{m}_* \in \mathbf{j}(I)$, so let $X = \{\hat{n} < \hat{n}_*: (\hat{m}_*, \hat{n}) \in Y\}$. By hypothesis we can find $\hat{X} \supseteq X$ with $\hat{X} \in [\hat{V}]^{\leq \hat{m}_*}$. Then $\hat{Y} := \{\hat{m}_*\} \times (\hat{X} \cap \hat{n}_*)$ is as desired.
\end{proof}

We can now wrap up the proof that $T_{nlow}$ is a minimal nonlow theory.  (B) implies (C) is due to Malliaris \cite{MalliarisFlex} in the context of regular ultrafilters on $\mathcal{P}(\lambda)$.

\begin{theorem}\label{nonlowInterp}
	Suppose $V \models ZFC^-$ is transitive, $\mathbf{j}: V \preceq \hat{V}$ is $\omega$-nonstandard, and $\lambda$ is given. Then the following are equivalent:
	
	\begin{itemize}
		\item[(A)] $\hat{V}$ $\lambda^+$-pseudosaturates $T_{nlow}$.
		\item[(B)] $\hat{V}$ $\lambda^+$-pseudosaturates some nonlow theory.
		\item[(C)] $\lambda < \lambda_{\hat{V}}(\Delta(IP))$ and $\lambda < \lambda_{\hat{V}}(\Delta(FDP))$.
	\end{itemize}
\end{theorem}
\begin{proof}
	(A) implies (B) is trivial.
	
	(B) implies (C): suppose $T \in V$ is nonlow; (B) is equivalent to $\lambda < \lambda_{\hat{V}}(T)$. Now $T$ is unstable, so $\lambda_{\hat{V}}(T) \leq \lambda_{\hat{V}}(\Delta(IP))$. If $T$ is unsimple, then $\lambda_{\hat{V}}(T) \leq \lambda_{\hat{V}}(\Delta(TP_2)) \leq \lambda_{\hat{V}}(\Delta(FDP))$, and if $T$ has the finite dividing property then $\lambda_{\hat{V}}(T) \leq \lambda_{\hat{V}}(\Delta(FDP))$ by Lemma~\ref{nlowLemma1}. Hence $\lambda_{\hat{V}}(T) \leq \lambda_{\hat{V}}(\Delta(FDP))$ in any case, and (C) holds.
	
	(C) implies (A): suppose $\lambda < \lambda_{\hat{V}}(\Delta(IP))$ and $\lambda < \lambda_{\hat{V}}(\Delta(FDP))$, and let $M \models T_{nlow}$ have universe $\omega$ (say), with $M \in V$. Write $M = (\omega, R, E, P, Q, U_n, P_n, Q_n^i, F_n: 1 \leq n < \omega, i < 2)$ and write $\mathbf{j}(M) = (\hat{\omega}, \hat{R}, \hat{E}, \hat{P}, \hat{Q}, \hat{U}_{\hat{n}}, \hat{Q}_{\hat{n}}^i, \hat{F}_{\hat{n}}:1 \leq \hat{n} < \hat{\omega})$ (so $\mathbf{j}_{\std}(M) = (\hat{\omega}, \hat{R}, \hat{E}, \hat{P}, \hat{Q}, \hat{U}_n, \hat{P}_n, \hat{Q}_n^i, \hat{F}_n: 1 \leq n < \omega, i < 2)$). In terms of our previous notation, we write, for instance, $\hat{U}_\omega = \hat{U} \backslash \bigcup_{n < \omega} \hat{U}_n$, a type-definable subset of $\mathbf{j}_{\std}(M)$ which is not definable in $\hat{V}$, and we write $\hat{U}_{\hat{\omega}} = \hat{U} \backslash \bigcup_{n < \omega} \hat{U}_{\hat{n}}$, a type-definable subset of $\mathbf{j}(M)$ in the sense of $\hat{V}$ which is not type-definable in $\mathbf{j}_{\std}(M)$.
	
	We show that $\mathbf{j}_{\std}(M)$ is $\lambda^+$-pseudosaturated.
	
	So let $p(x)$ be a pseudofinite partial type over $\mathbf{j}_{\std}(M)$ of cardinality at  most $\lambda$. We first of all claim that we can suppose $p(x)$ is a type over an algebraically closed set. Indeed, choose $\hat{n}_0 < \hat{\omega}$ such that $p(x)$ is over $\hat{n}_0$. Since algebraic closures of finite sets in $T_{nlow}$ are finite, we have that the algebraic closure of $\hat{n}_0$ in $\mathbf{j}(M)$, as computed in $\hat{V}$, is pseudofinite as desired. Let $C = \mbox{acl}_{\mathbf{j}_{\std}(M)}(\hat{n}_0)$; then $|C| \leq \lambda$, and further $C \subseteq \mbox{acl}_{\mathbf{j}(M)}(\hat{n}_0)$ is pseudofinite. Choose $\hat{n}_* < \hat{\omega}$ with $C \subseteq \hat{n}_*$. (If algebraic closures of finite sets were infinite, we would need to use overflow arguments instead.) Write $A = C \cap \hat{P}$ and write $B = C \cap \hat{Q}$. For each $1 \leq \hat{n} < \hat{\omega}$, let $A_{\hat{n}} = A \cap \hat{P}_{\hat{n}}$; for each $1 \leq \hat{n} < \hat{\omega}$ and for each $i < 2$, let $B_{\hat{n}}^i = B \cap \hat{Q}_{\hat{n}}^i$.
	
	We can suppose $p(x)$ is a complete nonalgebraic type over $C$. We can also suppose $p(x)$ is of one of the forms (I) through (IV) of Lemma~\ref{TCasLemma1}.
	
	Suppose first $p(x)$ is of form (I), say there are $1 \leq n < \omega$ and $X \subseteq [B_n^0]^{n-1} \times B^1_n$ such that $p(x) = p^I_X(x)$. Since $\lambda < \lambda_{\hat{V}}(\Delta(IP))$, we can find disjoint $\hat{Y}_0, \hat{Y}_1 \subseteq [\hat{Q}_n^0]^{n-1} \times \hat{Q}_n^1$ with $X \subseteq \hat{Y}_0$ and with $[B_n^0]^{n-1} \times B_n^1 \backslash X \subseteq \hat{Y}_1$. Let $\hat{p}(x) \in \hat{V}$ be the partial type asserting that $Q^0_n(x)$, and $\lnot R(a, x)$ for each $a \in \hat{P}_n \cap \hat{n}_*$, and for every $(v, b) \in \hat{Y}_0$, $R(F_n(v \cup \{x\}, b))$ holds, and for every $(v, b) \in \hat{Y}_1$, $R(F_n(v \cup \{x\}, b))$ fails. Easily, $\hat{V}$ believes $\hat{p}(x)$ is a consistent finite type, and hence $\hat{p}(x)$ is realized. But $p(x) \subseteq \hat{p}(x)$, so $p(x)$ is realized as desired.
	
	If $p(x)$ is of form (II) or of form (III), or of the form $p^{IV}_{X, c}$ for some $c \in \hat{U}_{\hat{\omega}}$, then a similar argument works.

	The crucial case is when $p(x)$ is of the form $p(x) = p^{IV}_{X, c}(x)$ for some $c \in C \cap \hat{U}_{\hat{m}}$, with $\hat{m}< \hat{\omega}$ nonstandard. Write $X_i = X \cap \hat{Q}_{\hat{m}}^i$ for each $i < 2$. Since $\lambda < \lambda_{\hat{V}}(\Delta(FDP))$ we can find $\hat{X}_0 \in [\hat{Q}^0_{\hat{m}}]^{\hat{m}-1}$ with $X_0 \subseteq \hat{X}_0$. Also, we can find disjoint $\hat{Y}_0, \hat{Y}_1 \subseteq \hat{Q}^1_{\hat{m}}$ with $X_1 \subseteq \hat{Y}_0$ and $B_{\hat{m}}^1 \backslash X_1 \subseteq \hat{Y}_1$. Let $\hat{p}(x) \in \hat{V}$ be the partial type asserting that $x E c$ and $R(x, a)$ for each $a \in \hat{X}_0 \cup \hat{Y}_0$ and $\lnot R(x, a)$ for each $a \in \hat{Y}_1$. Easily, $\hat{V}$ believes $\hat{p}(x)$ is a consistent finite type, and hence $\hat{p}(x)$ is realized. But $p(x) \subseteq \hat{p}(x)$, so $p(x)$ is realized as desired.
\end{proof}

We immediately get the following corollary.

\begin{corollary}\label{nonlowMin}
	$T_{nlow}$ is a minimal nonlow theory in $\trianglelefteq^\times_{1}$. Thus, this also holds for $\trianglelefteq^\times_\kappa$ and $\trianglelefteq^*_\kappa$ for every $\kappa$, and also for $\trianglelefteq$.
\end{corollary}

\section{Hypergraphs Omitting Cliques}\label{KeislerNewTNK}

In this section, we analyze the  major class of examples of simple theories with interesting amalgamation properties. 

\begin{definition}
	For each $2 \leq k < n < \omega$, let $T_{n, k}$ be the theory of the random $k$-ary, $n$-clique free hypergraph.
\end{definition}

These were introduced by Hrushovksi \cite{Hrush}, who proved that each $T_{n, 2}$ is unsimple, in fact it has $SOP_2$ and so is maximal in Keisler's order. We shall mainly be interested in the case of $T_{n, k}$ for $k \geq 3$; these are simple, with forking given by equality. In \cite{InfManyClass}, Malliaris and Shelah prove that for all $k < k'-1$, $T_{k+1, k} \not \trianglelefteq T_{k'+1, k'}$ (note that they subtract one from the indices). In \cite{AmalgKeislerUlrich}, we show that this holds for all $k < k'$.

The following are the relevant patterns:

\begin{definition}
	Suppose $S \subseteq [I]^k$ for some $k$, and suppose $n > k$. Then let $\Delta_{n, k}(S)$ be the pattern on $[I]^{k-1}$, consisting of all $s \in [[I]^{k-1}]^{<\aleph_0}$ such that there is no $v \in [I]^{n-1}$ with $[v]^{k-1} \subseteq s$ and $[v]^k \subseteq S$.
\end{definition}

Clearly, then, if $S \subseteq [I]^k$ is $n$-clique free, then $(T_{n, k}, R(x, \overline{y}))$ admits $\Delta_{n, k}(S)$.

\begin{definition}
	For each $k \geq 2$ and for each $n > k$, let $S_{k}$ be a random $k$-ary graph on $\omega$, and let $\Delta_{n, k} = \Delta_{n, k}(R_{k})$. (Up to equivalence, this does not depend on the choice of $S_k$.)
\end{definition}

Note that for every index set $I$ and for every $R \subseteq [I]^k$ and for every $n > k$, $\Delta{n,k}(R)$ is an instance of $\Delta_{n, k}$. Also note that admitting $\Delta_{n, 2}$ implies $SOP_3$.

It is not immediate that $T_{n, k}$ admits $\Delta_{n, k}$, since we did not require $S_k$ to be triangle-free. Towards this, the following fact will be helpful.

\begin{definition}
	Suppose $\Delta$ is a pattern on $I$. For each $n < \omega$, let $\Delta^n$ be the pattern on $[I]^{\leq n}$ consisting of all $s \in [[I]^{\leq n}]^{<\aleph_0}$ such that $\bigcup s \in \Delta$.
\end{definition}

\begin{theorem}\label{ConjunctInstances}
	Suppose $\Delta$ is a pattern on $I$ and $1 \leq n < \omega$.
	\begin{enumerate}
		\item If $T$ is a complete countable theory, then $T$ admits $\Delta$ if and only if $T$ admits $\Delta^n$.
		\item If $V \models ZFC^-$ is transitive with $\Delta, I \in V$ and if $\mathbf{j}: V \preceq \hat{V}$ with $\hat{V}$ not $\omega$-standard, then $\lambda_{\hat{V}}(\Delta) = \lambda_{\hat{V}}(\Delta^n)$.
		\item If $\mathcal{U}$ is an ultrafilter on $\mathcal{B}$, then $\lambda_{\mathcal{U}}(\Delta) = \lambda_{\mathcal{U}}(\Delta^n)$.
	\end{enumerate}
\end{theorem}
\begin{proof}
	(1): Note that $\Delta$ is an instance of $\Delta^n$ (using $[I]^1 \subseteq [I]^{\leq n}$), so it suffices to show that if $T$ admits $\Delta$ then $T$ admits $\Delta^n$. Suppose $\phi(x, y)$ admits $\Delta$ (really $x, y$ could be tuples). Let $\overline{y} = (y_i: i < n)$ and let $\psi(x, \overline{y}) = \bigwedge_{i <n} \phi(x, y_i)$. Easily then $\psi(x, \overline{y})$ admits $\Delta^n$.
	
	(2): Since $\Delta$ is an instance of $\Delta^n$, by Lemma~\ref{EquivLemma} it suffices to show that $\lambda_{\hat{V}}(\Delta) \leq \lambda_{\hat{V}}(\Delta^n)$. Write $(\hat{I}, \hat{\Delta}, \hat{\Delta}^n) = \mathbf{j}(I, \Delta, \Delta^n)$, and suppose $X \subseteq \hat{\Delta}^n$ is pseudofinite and of cardinality less than $\lambda_{\hat{V}}(\Delta)$. Write $Y = \bigcup X$. Then $Y \subseteq \hat{\Delta}$ is pseudofinite and of cardinality less than $\lambda_{\hat{V}}(\Delta)$, so there is some $\hat{s} \in \hat{\Delta}$ with $Y \subseteq \hat{s}$. Then $\hat{s}^{\leq n} \in \hat{\Delta}^n$ satisfies $X \subseteq \hat{s}^{\leq n}$, as desired.
	
	(3): follows from (2) and Theorem~\ref{CardInvarTheoremUlt}.
\end{proof}

We then obtain the following:

\begin{lemma}\label{TnkLemma1}
	Suppose $k \geq 2$ and $n > k$. Let $S$ be any $k$-ary hypergraph on $I$; then there is $i_* < \omega$ and a $k$-ary, $n$-clique free hypergraph $S'$ on $I'$ such that $\Delta_{n, k}(S)$ is an instance of $\Delta_{n,k}(S')^{i_*}$. 
\end{lemma}

\begin{proof}
	Write $I'= \omega \times (n-1)$. Let $S'$ be the $k$-ary graph on $I'$ consisting of the set of all order-preserving injections from $v$ to $n-1$, for $v \in S$. Since $S'$ is $n$-clique free, it suffices to show that $\Delta_{n, k}(S)$ is an instance of $\Delta_{n, k}(S')^{i_*}$, for $i_* := \binom{n-1}{k-1}$. In fact, we will embed all of $\Delta_{n, k}(S)$ into $\Delta_{n, k}(S')^{i_*}$ at once.
	
	Given $v \in [\omega]^{k-1}$, define $F(v) \subseteq [\omega \times (n-1)]^{k-1}$ to be the set of all order-preserving injections from $v$ to $n-1$. Easily, then, for every $s \subseteq [I]^{k-1}$ finite, $s \in \Delta_{n, k}(S)$ if and only if $F[s] \in \Delta_{n, k}(S')^k$, i.e. $\bigcup F[s] \in \Delta_{n, k}(S')$.
	
\end{proof}

Thus each $T_{n, k}$ admits $\Delta_{n, k}$. Actually, more is true:

\begin{theorem}\label{PositiveHypergraph}
Suppose $2 \leq n \leq n'$. Then $T_{n', k}$ admits $\Delta_{n, k}$.
\end{theorem}
\begin{proof}
By Lemma~\ref{TnkLemma1}, $T_{n', k}$ admits $\Delta_{n', k}$, so by Theorem~\ref{ConjunctInstances}, it suffices to note that $\Delta_{n, k}$ is an instance of $\Delta_{n', k}^{i_*}$ for $i_*$ large enough. For this it suffices to show that if $S \subseteq [I]^k$, then there is some $S' \subseteq [I']^k$ such that $\Delta_{n, k}(S)$ is an instance of $\Delta_{n', k}(S')^{i_*}$. Write $I'= I \cup u_*$, where $u_*$ is a new $n'-n$-element set. Put $S' = S \cup \{v' \in [I']^k: v' \cap u_* \not= \emptyset\}$. Write $i_* = \binom{k-1 + (n'-n)}{k-1}$. We claim this works.

Indeed, suppose $v \in [I]^{k-1}$; then define $F(v) \subseteq [I']^{k-1}$ via $F(v) = [v \cup u_*]^{k-1}$. Easily, for all $s \in [[I]^{k-1}]^{<\aleph_0}$, $s \in \Delta_{n, k}(S)$ if and only if $F[s] \in \Delta_{n', k}(S)^{i_*}$, so this works.
\end{proof}

We now aim to prove that $T_{n, k}$ is the $\trianglelefteq$-minimal theory admitting $\Delta_{n, k}$. As a preliminary case, we have to show that if $T$ admits $\Delta_{n, k}$ then $T$ is unstable (this should be clear but we spell out the details in our formalism). 

\begin{lemma}\label{TnkLemma2}
	Suppose $n > k \geq 2$. Then $\Delta(IP)$ is an instance of $(\Delta_{n, k})^{i_*}$ for some $i_*$.
\end{lemma}
\begin{proof}
	By Theorem~\ref{PositiveHypergraph} it suffices to consider the case $n = k+1$; in this case we will be able to set $i_* = k-1$.
	
	Let $u_*$ be a $k-2$ element set.
	
	Let $S$ be the $k$-ary graph on $\omega \times 2 \cup u_*$ consisting of all $w \in [(\omega \times 2) \cup u_*]^{k}$ of the form $u_* \cup \{(n, 0), (n, 1)\}$, for some $n < \omega$. Then $\Delta_{k+1, k}(S)$ is an instance of $\Delta_{k+1, k}$, so it suffices to show that $\Delta(IP)$ is an instance of $\Delta_{k+1, k}(S)^{k-1}$.
	
	Given $(j, i) \in \omega \times 2$, define $F(j, i)$ to be the set of all $v \in [u_* \cup \{(j,0), (j, 1)\}]^{k-1}$ other than $u_* \cup \{(j, i)\}$. Then clearly, for any $s \subseteq \omega \times 2$ finite, $s \in \Delta(IP)$ if and only if $F[s] \in \Delta(S)^{k-1}$. 
\end{proof}

We thus obtain the following easily.

\begin{theorem}\label{TnkLemma3}
	Suppose $2 \leq k < \omega$, suppose $V \models ZFC^-$ is transitive, and suppose $\mathbf{j}: V \preceq \hat{V}$. Then $\hat{V}$ $\lambda^+$-pseudosaturates $T_{n, k}$ if and only if $\lambda < \lambda_{\hat{V}}(\Delta_{n, k})$. In particular, $T_{n, k}$ is a $\trianglelefteq^\times_1$-minimal theory admitting $\Delta_{n, k}$ (so this is true for the other interpretability orders as well).
\end{theorem}
\begin{proof}
	By Theorem~\ref{PositiveHypergraph}, if $\hat{V}$ $\lambda^+$-pseudosaturates $T_{n, k}$ then $\lambda < \lambda_{\hat{V}}(\Delta_{n, k})$.
	
	So suppose $\lambda < \lambda_{\hat{V}}(\Delta_{n, k})$. Let $M \models T_{n, k}$ have universe $\omega$, and let $p(x)$ be a pseudofinite partial type over $\mathbf{j}_{\std}(M)$ of cardinality at most $\lambda$. We can suppose $p(x)$ is a complete type over $A$, where $|A| \leq \lambda$ and $A \subseteq \hat{n}$ for some $\hat{n}< \hat{\omega}$.  Write $M = (\omega, R)$, write $\mathbf{j}_{\std}(M) = \mathbf{j}(M) = (\hat{\omega}, \hat{R})$. (We also use $R$ for the symbol in the language.)
	
	Let $X_0 = \{\overline{a} \in \hat{n}^{k-1}: R(x,  \overline{a}) \in p(x)\}$ and let $X_1 = \{\overline{a} \in \hat{n}^{k-1}: \lnot R(x, \overline{a}) \in p(x)\}$. Since $[X_0]^{<\aleph_0} \subseteq \mathbf{j}(\Delta_{n, k}(R))$ we can find $\hat{X}_0' \in \mathbf{j}(\Delta_{n, k}(R))$ with $X_0 \subseteq \hat{X}_0'$. By Lemma~\ref{TnkLemma2} we can find disjoint $\hat{X}_0, \hat{X}_1 \subseteq \hat{n}$ with $X_0 \subseteq \hat{X}_0$ and $X_1 \subseteq \hat{X}_1$; we can suppose $\hat{X}_0 = \hat{X}_0'$ (by replacing them with $\hat{X}_0 \cap \hat{X}_0'$).
	
	Let $q(x) \in \hat{V}$ be the pseudofinite partial type, defined in $\hat{V}$ via: $q(x) = \{R(x, \overline{a}): \overline{a} \in \hat{X}_0\} \cup \{\lnot R(x, \overline{a}): \overline{a} \in \hat{X}_1\}$. Clearly $p(x) \subseteq q(x)$ and $\hat{V}$ believes $q(x)$ is a consistent finite type, so $q(x)$ must be realized in $\mathbf{j}(M)$ and we are done.
\end{proof}

\begin{corollary}
	For all $n' \geq n \geq k$, $T_{n, k} \trianglelefteq^\times_1 T_{n', k}$.
\end{corollary}

Before moving on, we show that if $T$ admits $\Delta_{n, k}$ then it does so in a particularly nice way. Some definitions are in order:

\begin{definition}
Suppose $n > k \geq 2$. Let $\mathcal{L}^-_{n, k}$ be the language $\{<, S\}$, where $S$ is $k$-ary, and let $\mathbf{K}_{n, k}^-$ be the class of all finite $\mathcal{L}$-structures $(M, <^M, S^M)$ where $<^M$ is a linear order and where $S^M \subseteq [M]^k$ (i.e. is irreflixive and symmetric). Let $\mathcal{L}_{n, k} = \mathcal{L}^-_{n, k} \cup \{Q\}$, where $Q$ is $k-1$-ary, and let $\mathbf{K}_{n, k}$ be the class of all finite substructures $(M, <^M, S^M, Q^M)$, where $<^M$ is a linear order, and $S^M \subseteq [M]^k$, and $Q^M \subseteq [M]^{k-1}$, and there is no $w \in [M]^{n-1}$ with $[w]^k \subseteq S$ and $[w]^{k-1} \subseteq Q$.

Clearly, $\mathbf{K}_{n, k}$ and $\mathbf{K}_{n, k}^-$ are Fraisse classes, and hence we can let $\tilde{T}_{n, k}$ and $\tilde{T}_{n, k}^-$ be the limit theories.
\end{definition}

In the following theorem, given structures $M, N$, we let $\binom{N}{M}$ be the set of all substructures of $N$ which are isomorphic to $M$.

\begin{theorem}\label{NesetrilRodl}
$\mathbf{K}_{n, k}$ and $\mathbf{K}_{n, k}^-$ are both Ramsey classes---that is, letting $\mathbf{K}$ be either of these classes, then whenever $A \subseteq B \in \mathbf{K}$, and whenever $n < \omega$, there is some $C \in K$ such that whenever $c: \binom{C}{A} \to n$ is a coloring, there is some $B' \in \binom{C}{B}$ such that $c$ is constant on $\binom{B'}{A}$.
\end{theorem}
\begin{proof}
This is a special instance of the Ne\v{s}it\v{r}il-R\"{o}dl theorem \cite{NesitrilRodl}.
\end{proof}

We can apply this to our situation in the following manner:

\begin{theorem}
Suppose $n > k \geq 2$. $T$ is a complete countable formula, and suppose $\phi(x, y)$ is a formula of $T$ (where possibly $x, y$ are tuples) which admits $\Delta_{n, k}$. Let $\mathfrak{C}$ be the monster model of $T_{n, k}$. 

Then for any $M \in \mathbf{K}_{n, k}$ we can find a sequence $(b_u: u \in [M]^{k-1})$ from $\mathfrak{C}$, and some $a \in \mathfrak{C}$, such that, writing $M^- := M \restriction_{\mathcal{L}^-_{n, k}} \in \mathbf{K}^-_{n, k}$:

\begin{itemize}
	\item For every $w \in [M]^{<\aleph_0}$, $tp_{\mathfrak{C}}(b_u: u \in [w]^{k-1})$ depends only on $tp_{M^-}(w)$ (for this to make sense, we are using that $w$ has a canonical enumeration from $<^M$, and hence also $(b_u: u \in [w]^{k-1})$ can be enumerated unambiguously by fixing an ordering of $[|w|]^{k-1}$);
	\item For every $w \in [M]^{<\aleph_0}$, $tp_{\mathfrak{C}}(a/(b_u: u \in [w]^{k-1}))$ depends only on $tp_M(w)$;
	\item For every $u \in Q^M$, $\mathfrak{C} \models \phi(a, b_u)$;
	\item For every $s \in [[M^{k-1}]^{<\aleph_0}]$, $\{\phi(x, b_u: u \in s)\}$ is consistent if and only if there is no $w \in [M]^{n-1}$ such that $[w]^k \subseteq S^M$ and $[w]^{k-1} \subseteq s$.
\end{itemize}
\end{theorem}
\begin{proof}
This follows easily from Theorem~\ref{NesetrilRodl} and compactness.
\end{proof}

We deduce two interesting corollaries from this. The following corollary is likely not optimal:

\begin{corollary}
	Suppose $k, k' \geq 2$. Then $T_{k'+1, k'}$ admits $\Delta_{k+1, k}$ if and only if $k= k'$.
	
	More generally, suppose $n > k > 2$ and $n' > k' \geq 2$, and either $\binom{n'-1}{k'-1} < \binom{n-1}{k-1}$ or else $\binom{n-1}{k-1} < k'$. Then $T_{n', k'}$ does not admit $\Delta_{n, k}$.
\end{corollary}
\begin{proof}
	Note that the first statement follows, since $T_{n, k}$ always admits $\Delta_{n, k}$ (by Lemmas~\ref{ConjunctInstances} and \ref{TnkLemma1}), and when $k = 2$ and $k' >2$ then $T_{k'+1, k'}$ has $NSOP_3$ and hence cannot admit $T_{k+1, k}$.

	Suppose towards a contradiction that $n > k > 2$ and $n' > k' \geq 2$, and either $\binom{n'-1}{k'-1} < \binom{n-1}{k-1}$ or else $\binom{n-1}{k-1} < k'$, and yet $T_{n', k'}$ admits $\Delta_{n, k}$.

	Let $\phi(\overline{x}, \overline{y})$ be a formula of $T_{n', k'}$ admitting $\Delta_{n, k}$. Let $(\mathfrak{C}, R^{\mathfrak{C}})$ be the monster model of $T_{n, k}$ and let $M \model \tilde{T}_{n, k}$; write $M^- = M \restriction_{\mathcal{L}^-_{n, k}} \model \tilde{T}_{n, k}^-$. Then we can find $(\overline{b}_u: u \in [M]^{k-1})$ from $\mathfrak{C}^{|\overline{y}|}$ and $\overline{a} \in \mathfrak{C}^{|\overline{x}|}$ as in Theorem~\ref{NesetrilRodl}.
	
	Let $q(\overline{y}) = tp(\overline{b}_u)$ for some or any $u \in [M]^{k-1}$. Let $p(\overline{x}, \overline{y}) = tp(\overline{a}, \overline{b}_u)$ for some or any $u \in Q^M$, so $p(\overline{x}, \overline{y})$ extends $q(\overline{y})$. Choose $w \in S^M$. Let $r(\overline{x}, \overline{b}_u: u \in [w]^{k-1}) := \bigcup \{p(\overline{x}, \overline{b}_u): u \in [w]^{k-1}\}$.

	Note that $r(\overline{x})$ is inconsistent, since $w \in S^M$ and $\phi(\overline{x}, \overline{b}_u) \in p(\overline{x}, \overline{b}_u)$ for each $u \in [w]^{k-1}$. Nonetheless, I claim that for every proper $X \subseteq [w]^{k-1}$, $\bigcup \{p(\overline{x}, \overline{b}_u): u \in X\}$ is consistent; suppose towards a contradiction this failed. Choose $(v(u): u \in X)$ from $Q^M$, such that $tp_{M^-}(v(u): u \in X) = tp_{M^-}(u: u \in X)$. Then $\bigcup \{p(\overline{x}, \overline{b}_{v(u)}): u \in X\}$ is inconsistent; but $\overline{a}$ realizes it.
	
	In particular, for all $u_0, u_1 \in [w]^{k-1}$, $p(\overline{x}, \overline{b}_{u_0}) \cup p(\overline{x}, \overline{b}_{u_1})$ is consistent (using $k > 2$). Hence the only way for $r(\overline{x})$ to be inconsistent is for it to create an $n'$-clique. In other words, we must be able to find $\overline{y} = (y_i: i < i_*)$ from $\overline{x}$ and $r = \{a_j: j < j_*\} \in [\mathfrak{C}]^{j_*}$, such that $i_* + j_* = n'+1$, and such that $r$ is an $R^{\mathfrak{C}}$-clique, and for every $t \in [r]^{<k'}$, there is some $u(t) \in [w]^{k-1}$ such that $t \subseteq u(t)$ and such that $p(\overline{x}, \overline{b}_{u(t)}) \models ``\overline{y}\cup t$ is an $R$-clique$''\}$. 
	
	Note that if $|r| < k'$ then this implies $p(\overline{x}, u(r))$ is inconsistent, a contradiction. So $|r| \geq k'$.

	Suppose first $\binom{n'-1}{k'-1} < \binom{n-1}{k-1}$. Let $X = \{u(t): t \in [r]^{k'-1}\}$. Then $X \subsetneq [w]^{k-1}$ and yet $\bigcup \{p(\overline{x}, b_u): u \in X\}$ is inconsistent, contradiction. 
	
	Finally, suppose $\binom{n-1}{k-1} < k'$. For each $u \in [w]^{k-1}$, choose $j_{u} < j_*$ such that $a_{j_{u}} \not \in b_u$ (possible as otherwise $p(\overline{x}, \overline{b}_u)$ would be inconsistent). But then $\{a_{j_u}: u \in [w]^{k-1}\}$ cannot be covered by any $\overline{b}_u$; since $\{a_{j_{u}}: u \in [w]^{k-1}\}$ has cardinality at most $k'-1$, this is a contradiction. 
	
	%
\end{proof}

Also, the following corollary will be helpful in \cite{AmalgKeislerUlrich}:

\begin{corollary}
Suppose $T$ is a countable simple theory and $\phi(x, y)$ is a formula of $T$ which admits $\Delta_{n, k}$ (so $n > k \geq 3$, since $T$ is simple). Let $\mathfrak{C}$ be the monster model of $T$. Then for every index set $I$ and for every $S \subseteq [I]^k$, we can find some countable $N \preceq \mathfrak{C}$ and some $(b_u: u \in [I]^{k-1})$ from $\mathfrak{C}$, such that for all $s \in [[I]^{k-1}]^{<\aleph_0}$:

\begin{itemize}
	\item If there is no $w \in [I]^{n-1}$ such that $[w]^{k} \subseteq S$ and $[w]^{k-1} \subseteq s$, then $\{\phi(x, b_u): u \in [w]^{k-1}\}$ does not fork over $N$;
	\item Otherwise, $\{\phi(x, b_u): u \in [w]^{k-1}\}$ is inconsistent.
\end{itemize}
\end{corollary}
\begin{proof}
Let $M \models \tilde{T}_{n, k}$ be $|I|^+$-saturated and let $M^- = M \restriction_{\mathcal{L}^-_{n, k}}$. Let $\mathfrak{C}$ be the monster model of $T_{n, k}$. Choose $(b_u: u \in [M]^{k-1})$ and $a$ from $\mathfrak{C}$ as in Theorem~\ref{NesetrilRodl}.

Choose $M_0 \preceq M$ countable such that $tp(a / b_u: u \in [M]^{k-1})$ does not fork over $(b_u: u \in [M_0]^{k-1})$, and let $N \preceq \mathfrak{C}$ be a countable elementary substructure containing $(b_u: u \in [M_0]^{k-1})$. By the saturation hypothesis on $M$, there is an embedding of $(I, S)$ into $(M \backslash M_0, S \cap [M \backslash M_0]^{k})$. Hence, it suffices to show that for all $s \in [[M \backslash M_0]^{k-1}]^{<\aleph_0}$, if there is no $w \in [M \backslash M_0]^{n-1}$ such that $[w]^k \subseteq S^M$ and $[w]^{k-1} \subseteq s$, then $\{\phi(x, b_u): u \in [w]^{k-1}\}$ does not fork over $N$. Suppose $w$ is given as such. Write $v = \bigcup s$. Since $M$ is $\aleph_1$-saturated, we can find some $v' \in [M \backslash M_0]^{<\aleph_0}$ such that $tp_{M^-}(v/ M_0) = tp_{M^-}(v'/M_0)$ and such that, if $f: v \to v'$ is the unique order-preserving bijection, then for all $u \in s$, $f[u] \in Q^M$. Then $\{\phi(x, b_{f[u]}): u \in s\}$ does not fork over $N$, since this is a subset of $tp(a/b_u: u \in [M]^{k-1})$. Hence $\{\phi(x, b_u): u \in s\}$ does not fork over $N$.
\end{proof}
\section{The Analysis in Terms of Ultrafilters}\label{FullBVModelsSec}

We phrase what we have done so far in terms of ultrafilters. 

\begin{definition}
	
	Given an index set $I$ and a complete Boolean algebra $\mathcal{B}$, an $I$-distribution in $\mathcal{B}$ is a function $\mathbf{A}: [I]^{<\aleph_0} \to \mathcal{B}_+$ (i.e. the positive elements of $\mathcal{B}$), such that $\mathbf{A}(\emptyset) = 1$, and $s \subseteq t$ implies $\mathbf{A}(s) \geq \mathbf{A}(t)$. If $\mathcal{D}$ is a filter on $\mathcal{B}$, we say that $\mathbf{A}$ is in $\mathcal{D}$ if $\mbox{im}(\mathbf{A}) \subseteq \mathcal{D}$.

	If $\mathbf{A}, \mathbf{B}$ are $I$-distributions in $\mathcal{B}$, then say that $\mathbf{B}$ refines $\mathbf{A}$ if $\mathbf{B}(s) \leq \mathbf{A}(s)$ for all $s \in [I]^{<\aleph_0}$.  Say that $\mathbf{A}$ is multiplicative if for all $s \in [I]^{<\aleph_0}$, $\mathbf{A}(s) = \bigwedge_{i \in s} \mathbf{A}(\{i\})$.
	
	Suppose $I, J$ are index sets, $\mathcal{B}$ is a complete Boolean algebra, $\mathbf{A}$ is a $J$-distribution and $\Delta$ is a pattern on $I$. Then say that $\mathbf{A}$ is a $(J, \Delta)$-distribution if for every $s \in [J]^{<\aleph_0}$ and for every $\mathbf{c} \in \mathcal{B}_+$ such that $\mathbf{c}$ decides $\mathbf{A}_{t}$ for all $t \subseteq s$, there is some $f: s \to I$ such that for all $t \subseteq s$, $\mathbf{c} \leq \mathbf{A}_t$ if and only if $f[t] \in \Delta$. If $T$ is a complete countable theory and $\phi(\overline{x}, \overline{y})$ is a formula of $T$, then say that $\mathbf{A}$ is a $(J, T, \phi)$-{\L}o{\'s} map if $\mathbf{A}$ is a $(J, \Delta_\phi)$-distribution.
	
	Suppose $\mathcal{U}$ is an ultrafilter on the complete Boolean algebra $\mathcal{B}$, and $\Delta$ is a pattern on $I$. Then let $\lambda_{\mathcal{U}}(\Delta)$ be the least $\lambda$ such that there is some $(\lambda, \Delta)$-distribution in $\mathcal{U}$ with no multiplicative refinement in $\mathcal{U}$. If $\phi(\overline{x}, \overline{y})$ is a formula of $T$ then let $\lambda_{\mathcal{U}}(T, \phi) = \lambda_{\mathcal{U}}(\Delta_\phi)$. In other words, $\lambda_{\mathcal{U}}(T,\phi)$ is the least $\lambda$ such that there is some $(\lambda, T, \phi)$-{\L}o{\'s} map $\mathbf{A}$ in $\mathcal{U}$ with no multiplicative refinement in $\mathcal{U}$. 
	
	Let $\lambda_{\mathcal{U}}(T)$ be the least infinite cardinal $\lambda$ such that $\mathcal{U}$ does not $\lambda^+$-saturates $T$. Then always $\lambda_{\mathcal{U}}(T) \geq \aleph_1$, and possibly $\lambda_{\mathcal{U}}(T) = \infty$. Further, $\mathcal{U}$ always $\lambda_{\mathcal{U}}(T)$-saturates $T$.
	
\end{definition}

\begin{remark}\label{LosRemark}
	In \cite{BVModelsUlrich}, we defined the notion of a $(\lambda,T, \overline{\phi})$-{\L}o{\'s} map for sequences of formulas $(\phi_\alpha(\overline{x}, \overline{y}_\alpha): \alpha < \lambda)$, where the $\overline{y}_\alpha$'s are all disjoint from each other and from $\overline{x}$. The above definition corresponds to the special case where each $\phi_\alpha(\overline{x}, \overline{y}_\alpha) = \phi(\overline{x}, \overline{y}_\alpha)$ for some fixed formula $\phi(\overline{x}, \overline{y})$. We also defined that a $\mathbf{A}$ is a $(\lambda, T)$-{\L}o{\'s} map if it is a $(\lambda, T, \overline{\phi})$-{\L}o{\'s} map for some $\overline{\phi}$.
\end{remark}

We have the following important but straightforward consequences of Theorem~\ref{Compactness}.
\begin{lemma}\label{CompactnessPatternsLemma1}
	Suppose $\mathcal{U}$ is an ultrafilter on the complete Boolean algebra $\mathcal{B}$, suppose $V \models ZFC^-$ is transitive, suppose $\mathbf{i}: V \preceq \mathbf{V}$ is $\lambda^+$-saturated, and suppose $\Delta \in V$ is a pattern on $I$. Finally, suppose $\mathbf{A}$ is a $J$-distribution in $\mathcal{B}$. Then $\mathbf{A}$ is a $(J, \Delta)$-distribution if and only if there are $(\mathbf{a}_j:j \in J)$ from $\mathbf{i}_{\std}(I)$, such that for all $s \in [J]^{<\aleph_0}$, $\|\{\mathbf{a}_j: j \in s\} \in \mathbf{i}(\Delta)\|_{\mathbf{V}} = \mathbf{A}(s)$. 
\end{lemma}
\begin{proof}
	First, suppose $\mathbf{A}$ is a $(J, \Delta)$-distribution. Choose $\mathbf{V}_0 \preceq \mathbf{V}$ with $|\mathbf{V}_0| \leq \lambda$, such that $\mathbf{i}(I), \mathbf{i}(\Delta) \in \mathbf{V}_0$. 
	
	We claim that Theorem~\ref{Compactness} implies there is $\mathbf{V}_1 \succeq \mathbf{V}_0$ and $\mathbf{a}_j \in \mathbf{V}_1$ for each $j \in J$, such that each $\|\mathbf{a}_j \in \mathbf{i}(I)\|_{\mathbf{V}_1} = 1$, and such that for all $s \in [J]^{<\aleph_0}$, $\|\{\mathbf{a}_j: j \in s\} \in \mathbf{i}(\Delta)\|_{\mathbf{V}_1} = \mathbf{A}(s)$. Indeed, let $\Gamma = \mathcal{L}(\mathbf{V}_0) \cup \{``\mathbf{a}_j \in \mathbf{i}(I)": j \in J\} \cup \{``\{\mathbf{a}_j: j \in s \} \in \mathbf{i}(\Delta)": s \in [J]^{<\aleph_0}\}$. We define $F_0 = F_1 = F: \Gamma \to \mathcal{B}$ via $F \restriction_{\mathcal{L}(\mathbf{V}_0)} = \|\cdot\|_{\mathbf{V}_0}$, and $F(``\mathbf{a}_j \in \mathbf{i}(I)") = 1$, and $F(``\{\mathbf{a}_j: j \in s\} \in \mathbf{i}(\Delta)") = \mathbf{A}(s)$. Then clearly Theorem~\ref{Compactness} applies and gives $\mathbf{V}_1$ as desired.
	
	Since $\mathbf{V}$ is $\lambda^+$-saturated, we can in fact choose such $(\mathbf{a}_j: j \in J)$ from $\mathbf{V}$. 
	
	Conversely, suppose $(\mathbf{a}_j: j \in J)$ are given. Suppose $s \in [J]^{<\aleph_0}$ and $\mathbf{c} \in \mathcal{B}_{+}$ are given, such that $\mathbf{c}$ decides $\mathbf{A}_t$ for every $t \subseteq s$. Let $\Delta_0 = \{t \subseteq s: \mathbf{c} \leq \mathbf{A}_t\}$, a pattern on $s$. We need to find $\{i_j: j \in s\}$ such that for all $t \subseteq s$, $\{i_j: j \in t\} \in \Delta$ if and only if $t \in \Delta_0$. Suppose towards a contradiction this were impossible. Then by elementarity of $V \preceq \mathbf{V}$, we would have $\|$there exists $(\mathbf{b}_j: j \in J)$ from $\mathbf{i}(I)$ such that for all $t \subseteq s$, $\{\mathbf{b}_j: j \in t\} \in \mathbf{i}(\Delta)$ if and only if $t \in \Delta_0\|_{\mathbf{V}} = 0$. But $0 < \mathbf{c} \leq \|(\mathbf{a}_j: j \in J)$ is a sequence from $\mathbf{i}(I)$ and for all $t \subseteq s$, $\{\mathbf{a}_j: j \in t\}\in \mathbf{i}(\Delta)$ if and only if $t \in \Delta_0\|_{\mathbf{V}}$, contradiction.
\end{proof}
\begin{lemma}\label{CompactnessPatternsLemma2}
	Suppose $\mathcal{U}$ is an ultrafilter on the complete Boolean algebra $\mathcal{B}$, suppose $V \models ZFC^-$ is transitive, suppose $\mathbf{i}: V \preceq \mathbf{V}$ is $\lambda^+$-saturated, and suppose $\Delta \in V$ is a pattern on $I$. Suppose $\mathbf{A}$ is a $(J, \Delta)$-distribution in $\mathcal{B}$ as witnessed by $(\mathbf{a}_j: j \in J)$, i.e. each $\mathbf{a}_j \in \mathbf{i}_{\std}(I)$, and for all $s \in [J]^{<\aleph_0}$, $\|\{\mathbf{a}_j: j \in s\} \in \mathbf{i}(\Delta)\|_{\mathbf{V}} = \mathbf{A}(s)$. Finally, suppose $\mathbf{B}$ is a $J$-distribution in $\mathcal{B}$. Then $\mathbf{B}$ is a multiplicative refinement of $\mathbf{A}$ if and only if there is some $\mathbf{Y} \in \mathbf{i}_{\std}(\Delta)$ such that for all $s \in [J]^{<\aleph_0}$, $\|\{\mathbf{a}_j: j \in s\} \subseteq \mathbf{Y}\|_{\mathbf{V}} = \mathbf{B}(s)$.
\end{lemma}
\begin{proof}
	First, suppose $\mathbf{B}$ is a multiplicative refinement of $\mathbf{A}$. Choose $\mathbf{V}_0 \preceq \mathbf{V}$ with $|\mathbf{V}_0| \leq \lambda$, such that $\mathbf{i}(I), \mathbf{i}(\Delta)$ and each $\mathbf{a}_j \in \mathbf{V}_0$. 
	
	We claim that Theorem~\ref{Compactness} implies there is $\mathbf{V}_1 \succeq \mathbf{V}_0$ and $\mathbf{Y} \in \mathbf{V}_1$ such that $\| \mathbf{Y} \in \mathbf{i}(\Delta) \|_{\mathbf{V}} = 1$, and for all $s \in [J]^{<\aleph_0}$, $\|\{\mathbf{a}_j: j \in s\} \subseteq \mathbf{Y}\|_{\mathbf{V}} = \mathbf{B}(s)$. Indeed, let $\Gamma = \mathcal{L}(\mathbf{V}_0) \cup \{``\mathbf{Y} \in \mathbf{i}(\Delta)"\} \cup \{``\{\mathbf{a}_j: j \in s\} \subseteq \mathbf{Y}": s \in [J]^{<\aleph_0}\}$. We define $F_0 = F_1 = F: \Gamma \to \mathcal{B}$ via $F \restriction_{\mathcal{L}(\mathbf{V}_0)} = \|\cdot\|_{\mathbf{V}_0}$, and $F(\mathbf{Y} \in \mathbf{i}(\Delta)) = 1$, and $F(\{\mathbf{a}_j: j \in s\} \subseteq \mathbf{Y}) = \mathbf{B}(s)$. Then clearly Theorem~\ref{Compactness} applies and gives $\mathbf{V}_1$ as desired.
	
	Since $\mathbf{V}$ is $\lambda^+$-saturated, we can in fact choose such $\mathbf{Y} \in \mathbf{V}$.
	
	Conversely, suppose $\mathbf{Y}$ is given. Trivially, $\mathbf{B}$ is multiplicative. Furthermore, since each $\mathbf{B}(s) = \|\{\mathbf{a}_j: j \in s \} \subseteq \mathbf{Y}\|_{\mathbf{V}}$ and $\|\mathbf{Y} \in \mathbf{i}(\Delta)\|_{\mathbf{V}} = 1$, we have that $\mathbf{B}(s) \leq \|\{\mathbf{a}_j: j \in s\} \in \mathbf{i}(\Delta)\|_{\mathbf{V}} = \mathbf{A}(s)$, so $\mathbf{B}$ refines $\mathbf{A}$.
\end{proof}

As a first example of how these lemmas are used, we have the following.  It strengthens a theorem from \cite{BVModelsUlrich}, which states that $\mathcal{U}$ $\lambda^+$-saturates $T$ if and only if every $(\lambda, T)$-{\L}o{\'s} map in $\mathcal{U}$ has a multiplicative refinement in $\mathcal{U}$.

\begin{theorem}\label{CardInvarTheoremUlt}
	Suppose $\mathcal{U}$ is an ultrafilter on the complete Boolean algebra $\mathcal{B}$, suppose $V \models ZFC^-$ is transitive, suppose $\mathbf{i}: V \preceq \mathbf{V}$ is $\lambda^+$-saturated, and suppose $\Delta \in V$ is a pattern on $I$. Then $\lambda < \lambda_{\mathcal{U}}(\Delta) $ if and only if $\lambda < \lambda_{\mathbf{V}/\mathcal{U}}(\Delta)$. Thus, for all complete countable theories $T$, $\lambda_{\mathcal{U}}(T)$ is the minimum of all $\lambda_{\mathcal{U}}(T, \phi(\overline{x}, \overline{y}))$, for $\phi(\overline{x}, \overline{y})$ a formula of $T$; this is the same as the minimum of all $\lambda_{\mathcal{U}}(T, \phi(x, \overline{y}))$ for all formulas $\phi(x, \overline{y})$ of $T$ with $x$ a single variable.
\end{theorem}
\begin{proof}
	
	Write $\hat{V} = \mathbf{V}/\mathcal{U}$ and let $\mathbf{j}: V \preceq \hat{V}$ be the composition $[\cdot]_{\mathcal{U}} \circ \mathbf{i}$. We know that $\mathcal{U}$ $\lambda^+$-saturates $T$ if and only if $\hat{V}$ $\lambda^+$-pseudosaturates $T$ if and only if $\lambda < \lambda_{\hat{V}}(T)$, and since $\hat{V}$ is $\aleph_1$-saturated, we know by Theorem~\ref{CardInvarTheoremInterp} that $\lambda_{\hat{V}}(T) = \lambda_{\hat{V}}^{loc}(T) = \lambda_{\hat{V}}^{loc, 1}(T)$. Thus it suffices to show the first claim. So suppose $\Delta \in V$ is a pattern on $I$.
	
	Suppose first $\lambda \geq \lambda_{\mathcal{U}}(\Delta)$. Then we can find a $(\lambda, \Delta)$-distribution in $\mathcal{U}$ with no multiplicative refinement in $\mathcal{U}$; call it $\mathbf{A}$. By Lemma~\ref{CompactnessPatternsLemma1}, we can choose $(\mathbf{a}_\alpha: \alpha < \lambda)$ from $\mathbf{i}_{\std}(I)$, such that for all $s \in [\lambda]^{<\aleph_0}$, $\|\{\mathbf{a}_\alpha: \alpha \in s\} \in \mathbf{i}(\Delta)\|_{\mathbf{V}} = \mathbf{A}(s)$. Write $\hat{a}_\alpha = [\mathbf{a}_\alpha]_{\mathcal{U}}$ and write $X = \{\hat{a}_\alpha: \alpha < \lambda\} \subseteq \mathbf{j}(I)$. Since $\mathbf{V}$ is $\lambda^+$-saturated, an easy application of Theorem~\ref{Compactness} shows that $X$ is pseudofinite (this is Lemma 4.8 of \cite{InterpOrdersUlrich}). It suffices to show there is no $\hat{Y} \in \mathbf{j}(\Delta)$ with $\hat{X} \subseteq \hat{Y}$; suppose towards a contradiction there was some such $\hat{Y}$. Write $\hat{Y} = [\mathbf{Y}']_{\mathcal{U}}$. Choose $\mathbf{Y} \in \mathbf{V}$ such that $\|\mbox{if } \mathbf{Y}' \in \mathbf{i}(\Delta) \mbox{ then } \mathbf{Y} = \mathbf{Y}', \mbox{ and otherwise } \mathbf{Y} = \emptyset\|_{\mathbf{V}} = 1$; then $[\mathbf{Y}]_{\mathcal{U}} = \hat{Y}$ and further $\|\mathbf{Y} \in \mathbf{i}(\Delta)\|_{\mathbf{V}} = 1$. Define $\mathbf{B}(s) = \|\{a_\alpha: \alpha \in s\} \subseteq \mathbf{Y}\|_{\mathbf{V}}$; by Lemma~\ref{CompactnessPatternsLemma2} this is a multiplicative refinement of $\mathbf{A}$ in $\mathcal{U}$, contradicting the choice of $\mathbf{A}$. Thus no such $\hat{Y}$ exists, and this witnesses $\lambda \geq \lambda_{\hat{V}}(\Delta)$.
	
	Conversely, if $\lambda \geq \lambda_{\hat{V}}(\Delta)$, then we can find some pseudofinite $X \subseteq \mathbf{j}(I)$ of size at most $\lambda$ with $[X]^{<\aleph_0} \subseteq \mathbf{j}(\Delta)$, such that there is no $\hat{Y} \in \mathbf{j}(\Delta)$ with $X \subseteq \hat{Y}$. Enumerate $X = \{[\mathbf{a}_\alpha]_{\mathcal{U}}: \alpha < \lambda\}$; by a similar trick as above, we can arrange that each $\|\mathbf{a}_\alpha \in \mathbf{i}(I)\|_{\mathbf{V}} = 1$. For each $s \in [\lambda]^{<\aleph_0}$ define $\mathbf{A}(s) = \|\{\mathbf{a}_\alpha: \alpha \in s\} \in \mathbf{i}(\Delta)\|_{\mathbf{V}}$. Then $\mathbf{A}$ is an $(I, \Delta)$-distribution in $\mathcal{U}$, by Lemma~\ref{CompactnessPatternsLemma1}. Suppose towards a contradiction that $\mathbf{B}$ were a multiplicative refinement of $\mathbf{A}$ in $\mathcal{U}$. By Lemma~\ref{CompactnessPatternsLemma2}, we can choose $\mathbf{Y} \in \mathbf{i}_{\std}(\Delta)$ such that for  all $s \in [\lambda]^{<\aleph_0}$, $\|\{\mathbf{a}_\alpha: \alpha \in s\} \in \mathbf{Y}\|_{\mathbf{V}} = \mathbf{B}(s)$; then $\hat{Y} = [\mathbf{Y}]_{\mathcal{U}}$ contradicts our choice of $X$.
\end{proof}

We now give characterizations of $\lambda_{\mathcal{U}}(\Delta(IP))$, $\lambda_{\mathcal{U}}(\Delta(TP_2))$ and $\lambda_{\mathcal{U}}(\Delta(FDP))$.

\begin{lemma}\label{IPLemma2}
	Suppose $\mathcal{B}$ is a complete Boolean algebra and $\mathcal{U}$ is an ultrafilter on $\mathcal{B}$. Write $I = \lambda \times 2$. Then $\lambda_{\mathcal{U}}(\Delta(IP))$ is the least $\lambda$ such that there is a $I$-distribution $\mathbf{A}$ in $\mathcal{U}$ of the following form, with no multiplicative refinement in $\mathcal{U}$. Namely, for some $\mathbf{V} \models^{\mathcal{B}} ZFC^-$ and for some $(\mathbf{a}_{\alpha, i}: (\alpha, i) \in I)$ from $\mathbf{V}$, we have that each $\mathbf{A}(s) =\bigwedge\{\|\mathbf{a}_{\alpha, 0} \not= \mathbf{a}_{\beta, 1}\|_{\mathbf{V}}: (\alpha, 0), (\beta, 1) \in s\}$.
\end{lemma}
\begin{proof}
	Easily, any such $\mathbf{A}$ is a $(\lambda \times 2, \Delta({IP}))$-distribution. Conversely, suppose $\mathbf{A}$ is a given $(\lambda, \Delta({IP}))$-distribution in $\mathcal{U}$. Choose some transitive $V \models ZFC^-$, choose $\mathbf{i}: V \preceq \mathbf{V}$ with $\mathbf{V}$ $\lambda^+$-saturated, and choose $(\mathbf{x}_\alpha: \alpha < \lambda)$ a pseudofinite sequence from $\mathbf{i}_{\std}(\omega \times \omega)$ such that for all $s \in [\lambda]^{<\aleph_0}$, $\|\{\mathbf{x}_\alpha: \alpha \in s\} \in \mathbf{i}(\Delta({IP}))\|_{\mathbf{V}} = \mathbf{A}(s)$; this is possible by Lemma~\ref{CompactnessPatternsLemma1}.  
	
	For each $\alpha < \lambda$, choose $\mathbf{m}_{\alpha}$ such that $\|\mathbf{x}_\alpha \in \{\mathbf{m}_\alpha\} \times 2\|_{\mathbf{V}} = 1$, and choose $\mathbf{k}_\alpha$ such that $\|\mathbf{x}_\alpha \in \omega \times \{\mathbf{k}_\alpha\}\|_{\mathbf{V}} = 1$; this is possible by fullness of $\mathbf{V}$ (note $\mathbf{k}_\alpha$ is determined by the pair $(\|\mathbf{k}_\alpha = 0\|_{\mathbf{V}}, \|\mathbf{k}_\alpha = 1\|_{\mathbf{V}})$. For each $\alpha < \lambda$, there is a unique $f(\alpha) < 2$ such that $\|\mathbf{k}_\alpha = f(\alpha)\|_{\mathbf{V}} \in \mathcal{U}$. For each $s \in [\lambda]^{<\aleph_0}$, let $\mathbf{A}'(s) = \mathbf{A}(s) \wedge \bigwedge_{\alpha \in s} \|\mathbf{k}_\alpha = f(\alpha)\|_{\mathbf{V}}$; then this is a conservative refinement of $\mathbf{A}$ in $\mathcal{U}$. Thus $\mathbf{A}'$ has a multiplicative refinement in $\mathcal{U}$ if and only if $\mathbf{A}$ does.
	
	For each $i < 2$, let $\{\mathbf{a}_{\alpha, i}: \alpha < \lambda\}$ list $\{\mathbf{m}_\beta: f(\beta) = i\}$, with repetitions if necessary. Let $\mathbf{A}''$ be the $\lambda \times 2$-distribution defined from $(\mathbf{a}_{\alpha, i}: \alpha < \lambda, i < 2)$ as in the statement of the lemma. Note that $\mathbf{A}''$ is in $\mathcal{U}$, since whenever $f(\beta) \not= f(\beta')$, we have $\|\mathbf{m}_\beta \not= \mathbf{m}_{\beta'}\|_{\mathbf{V}} \in \mathcal{U}$. Thus $\mathbf{A}''$ has a multiplicative refinement in $\mathcal{U}$, which easily gives a multiplicative refinement of $\mathbf{A}'$.
\end{proof}

\begin{corollary}\label{RandGraphUlt}
	Suppose $\mathcal{U}$ is an ultrafilter on the complete Boolean algebra $\mathcal{B}$. Then the following are equivalent:
	
	\begin{itemize}
		\item[(A)] $\mathcal{U}$ $\lambda^+$-saturates $T_{rg}$;
		\item[(B)] $\mathcal{U}$ $\lambda^+$-saturates some unstable theory;
		\item[(C)] $\lambda < \lambda_{\mathcal{U}}(\Delta(IP))$.
	\end{itemize}
\end{corollary}

It is a major open problem in the subject, see e.g. Problem (1) in the list of open problems in \cite{InterpOrders}, to determine the Keisler class of the random graph model-theoretically. Examples of theories in this class are rather sparse; for instance, one can show $n$-ary random hypergraphs are equivalent to $T_{rg}$, but the following concrete question remains open. Let $ACFA$ be the theory of an algebraically closed field of with a generic automorphism. $ACFA$ is incomplete; one must specify the characteristic, and also the isomorphism type of the automorphism restricted to the algebraic closure of the emptyset. 

\vspace{2 mm}

\noindent \textbf{Question.} Suppose $T$ is a completion of $ACFA$. Is $T$ equivalent to the random graph in $\trianglelefteq^\times_{1}$, or at least in $\trianglelefteq$?

(A) $\leq$ (C) of the following lemma is Lemma 6.8 of \cite{HypSeq}, in the special case where $\mathcal{U}$ is an ultrafilter on $\mathcal{P}(\lambda)$. The other inequalities are implicit in Section 6 there, although they are formulated differently. In \cite{GoodForEq}, Malliaris defines $\mathcal{U}$ to be $\lambda^+$-good for equality if $\lambda$ is less than the value in (C).

\begin{lemma}\label{TP2Lemma2}
	Suppose $\mathcal{B}$ is a complete Boolean algebra and $\mathcal{U}$ is an ultrafilter on $\mathcal{B}$. Then the following cardinals are equal:
	
	\begin{itemize}
		\item[(A)] $\lambda_{\mathcal{U}}(\Delta(TP_2))$;
		\item[(B)] The least $\lambda$ such that there are $\mathbf{V} \models^{\mathcal{B}} ZFC^-$ and $(\mathbf{a}_\alpha: \alpha < \lambda)$ from $\mathbf{V}$, such that there is no multiplicative $\lambda$-distribution $\mathbf{B}$ in $\mathcal{U}$ such that each $\mathbf{B}(\{\alpha, \beta\})$ decides $\|\mathbf{a}_\alpha = \mathbf{a}_\beta\|_{\mathbf{V}}$ (necessarily as dictated by $\mathcal{U}$);
		\item[(C)] The least $\lambda$ such that there are $\mathbf{V} \models^{\mathcal{B}} ZFC^-$ and $(\mathbf{a}_\alpha: \alpha < \lambda)$ from $\mathbf{V}$, such that $[\mathbf{a}_\alpha]_{\mathcal{U}} \not= [\mathbf{a}_\beta]_{\mathcal{U}}$ for all $\alpha \not= \beta$, and such that there is no multiplicative $\lambda$-distribution $\mathbf{B}$ in $\mathcal{U}$ with $\mathbf{B}(\{\alpha, \beta\}) \leq \|\mathbf{a}_\alpha \not= \mathbf{a}_\beta\|_{\mathbf{V}}$ for all $\alpha \not= \beta$.
	\end{itemize}
\end{lemma}

\begin{proof}
	Let $\lambda_A, \lambda_B, \lambda_C$ be the cardinals defined in items (A), (B), (C). 
	
	$\lambda_C \leq \lambda_B$: suppose $\lambda < \lambda_C$, we show $\lambda < \lambda_B$. Let $\mathbf{V}, (\mathbf{a}_\alpha: \alpha < \lambda)$ be given. Let $E$ be the equivalence relation on $\lambda$ defined via: $E(\alpha, \beta)$ if $[\mathbf{a}_\alpha]_{\mathcal{U}} = [\mathbf{a}_\beta]_{\mathcal{U}}$ (i.e. $\| \mathbf{a}_\alpha = \mathbf{a}_\beta\|_{\mathbf{V}} \in \mathcal{U}$). Let $I \subseteq \lambda$ be a choice of representative for $\lambda/E$, i.e., such that each $\alpha < \lambda$ is $E$-related to exactly one $\beta \in I$. Let $f: \lambda \to I$ be the map witnessing this, so for all $\alpha < \lambda$ and for all $\beta \in I$, $[\mathbf{a}_\alpha]_{\mathcal{U}} = [\mathbf{a}_\beta]_{\mathcal{U}}$ if and only if $\beta = f(\alpha)$. Since $\lambda < \lambda_C$, we can find a multiplicative $I$-distribution $\mathbf{B}_0$ in $\mathcal{U}$ with each $\mathbf{B}_0(\alpha, \beta) \leq \|\mathbf{a}_\alpha \not= \mathbf{a}_\beta\|_{\mathbf{V}}$. Define $\mathbf{B}$, a $\lambda$-distribution in $\mathcal{U}$, via $\mathbf{B}(s) = \mathbf{B}_0(f[s]) \wedge \bigwedge_{\alpha \in s} \|\mathbf{a}_\alpha = \mathbf{a}_{f(\alpha)}\|_{\mathbf{V}}$. Then $\mathbf{B}$ is clearly witnesses $\lambda < \lambda_B$.
	
	$\lambda_B \leq \lambda_A$: suppose $\lambda < \lambda_B$, we show $\lambda < \lambda_A$. Suppose $\mathbf{A}$ is a given $(\lambda, \Delta({TP}_2))$-distribution in $\mathcal{U}$. Choose some transitive $V \models ZFC^-$, and let $\mathbf{i}: V \preceq \mathbf{V}$ with $\mathbf{V}$ $\lambda^+$-saturated, and choose $(\mathbf{x}_\alpha: \alpha < \lambda)$ a sequence from $\mathbf{i}_{\std}(\omega \times \omega)$ such that for all $s \in [\lambda]^{<\aleph_0}$, $\|\{\mathbf{x}_\alpha: \alpha \in s\} \in \mathbf{i}(\Delta(TP_2))\|_{\mathbf{V}} = \mathbf{A}(s)$; this is possible by Lemma~\ref{CompactnessPatternsLemma1}.
	
	For each $\alpha < \lambda$, choose $\mathbf{n}_\alpha, \mathbf{m}_\alpha$ such that $\|\mathbf{x}_\alpha = (\mathbf{n}_\alpha, \mathbf{m}_\alpha)\|_{\mathbf{V}} = 1$ (possible by fullness of $\mathbf{V}$). By two applications of $\lambda < \lambda_B$, we can find a multiplicative distribution $\mathbf{B}$ in $\mathcal{U}$ such that for all $\alpha < \beta < \lambda$, $\mathbf{B}(\{\alpha, \beta\})$ decides $\|\mathbf{n}_\alpha = \mathbf{n}_\beta\|_{\mathbf{V}}$ and decides $\|\mathbf{m}_\alpha = \mathbf{m}_\beta\|_{\mathbf{V}}$, from which it follows that $\mathbf{B}$ is a multiplicative refinement of $\mathbf{A}$.

	$\lambda_A \leq \lambda_C$: suppose $\lambda < \lambda_A$, we show $\lambda < \lambda_C$. So suppose $\mathbf{V}, (\mathbf{a}_\alpha: \alpha < \lambda)$ are given. Define a $\lambda$-distribution $\mathbf{A}$ in $\mathcal{U}$ via $\mathbf{A}(s) = \bigwedge_{\alpha < \beta \in s} \| \mathbf{a}_\alpha \not= \mathbf{a}_\beta \|_{\mathbf{V}}$. Easily this is a $(\lambda, \Delta(TP_2))$-distribution, and thus it has multiplicative refinement $\mathbf{B}$ in $\mathcal{U}$. $\mathbf{B}$ is as desired.
\end{proof}

\begin{corollary}\label{StarRandGraphUlt}
	Suppose $\mathcal{U}$ is an ultrafilter on the complete Boolean algebra $\mathcal{B}$. Then the following are equivalent:
	
	\begin{itemize}
		\item[(A)] $\mathcal{U}$ $\lambda^+$-saturates $T_{rf}$;
		\item[(B)] $\mathcal{U}$ $\lambda^+$-saturates some unsimple theory;
		\item[(C)] $\lambda < \lambda_{\mathcal{U}}(\Delta(TP_2))$.
	\end{itemize}
\end{corollary}

The corresponding invariants for $\lambda_{\mathcal{U}}(\Delta(FDP))$ have been studied under various guises. $\lambda$-OK was first defined by Kunen \cite{KunenOK0}, and $\lambda$-flexibility was first defined by Malliaris in \cite{MalliarisFlex}. Previously these definitions were made only in the case of $\mathcal{B} = \mathcal{P}(\lambda)$.

\begin{definition}
	The ultrafilter $\mathcal{U}$ on the complete Boolean algebra $\mathcal{B}$ is $\lambda$-OK if whenever $\mathbf{A}$ is a $\lambda$-distribution in $\mathcal{U}$ such that for all $s, t \in [\lambda]^n$, $\mathbf{A}(s) = \mathbf{A}(t)$, we have that $\mathbf{A}$ has a multiplicative refinement in $\mathcal{U}$. (Note in this case that $\mathbf{A}$ is determined by the descending sequence $(\mathbf{A}(n): n < \omega)$.) $\mathcal{U}$ is $\lambda$-flexible if $\mathcal{U}$ is $\aleph_1$-incomplete and, for every $\mathcal{U}$-nonstandard $\mathbf{n} \in (\omega, <)^{\mathcal{B}}/\mathcal{U}$, there is some multiplicative $\lambda$-distribution $\mathbf{B}$ in $\mathcal{U}$, such that for all $s \in [\lambda]^n$, $\mathbf{B}(s) \leq \|\mathbf{n} \geq n\|_{(\omega, <)^{\mathcal{B}}}$.
\end{definition}

So clearly, if $\mathcal{U}$ is $\aleph_1$-complete, then $\mathcal{U}$ is $\lambda$-OK for all $\lambda$.

We first remark that this definition of $\lambda$-flexibility coincides with the one given by Malliaris.  We recall some definitions from \cite{BVModelsUlrich} and (independently) \cite{Parente1} \cite{Parente2}. Suppose $\mathcal{B}$ is a complete Boolean algebra, and $(\mathbf{a}_i: i \in I)$ is a sequence from $\mathcal{B}$. Then say that $(\mathbf{a}_i: i \in I)$ is $I$-regular if it has the finite intersection property, and infinite intersections are $0$, and the set of all $\mathbf{c} \in \mathcal{B}_+$ which decide $\mathbf{a}_i$ for every $i < I$ is dense in $\mathcal{B}_+$; this coincides with the usual definition when $\mathcal{B}$ is $\lambda^+$-distributive.  The $\lambda$-distribution $\mathbf{B}$ is $\lambda$-regular if $(\mathbf{B}(\{\alpha\}): \alpha < \lambda)$ is $\lambda$-regular, or equivalently $(\mathbf{B}(s): s \in [\lambda]^{<\aleph_0})$ is $[\lambda]^{<\aleph_0}$-regular. The filter $\mathcal{D}$ is $\lambda$-regular if it contains a $\lambda$-regular family.

The following is motivated by Mansfield's argument in Theorem 4.1 of \cite{BooleanUltrapowers}.

\begin{theorem}\label{FlexEquiv}
	Suppose $\mathcal{U}$ is an ultrafilter on $\mathcal{B}$, and $[\mathbf{n}]_{\mathcal{U}} \in (\omega, <)^{\mathcal{B}}/\mathcal{U}$ is nonstandard, and $\mathbf{B}$ is a multiplicative $\lambda$-distribution in $\mathcal{U}$. Then the following are equivalent:
	
	\begin{enumerate}
		\item[(A)] For every $s \in [\lambda]^n$, $\mathbf{B}(s) \leq \|\mathbf{n} \geq n\|_{(\omega,<)^{\mathcal{B}}}$.
		\item[(B)] $\mathbf{B}$ is $\lambda$-regular, and for every $\mathbf{c} \in \mathbf{B}$, if $\mathbf{c}$ decides $\mathbf{B}(s)$ for each $s \in [\lambda]^{<\aleph_0}$ (or equivalently, $\mathbf{c}$ decides $\mathbf{B}(\{\alpha\})$ for each $\alpha < \lambda$) and if we write $n := |\{\alpha < \lambda: \mathbf{c} \leq \mathbf{B}(\{\alpha\})|$, then $\mathbf{c} \leq \|\mathbf{n} \geq m\|_{(\omega, <)^{\mathcal{B}}}$.
	\end{enumerate}
\end{theorem}
In particular, if $\mathcal{U}$ is $\lambda$-flexible, then $\mathcal{U}$ is $\lambda$-regular.
\begin{proof}
	(A) implies (B): We show that for every $\mathbf{c} \in \mathcal{B}$ nonzero, there is $\mathbf{c}' \leq \mathbf{c}$ nonzero such that $\mathbf{c}'$ decides each $\mathbf{B}(\{\alpha\})$, and there is $m_* < \omega$ such that $\mathbf{c}' \leq \|\mathbf{n} = m_*\|_{(\omega,<)^{\mathcal{B}}}$ and $|\{\alpha < \lambda: \mathbf{c} \leq \mathbf{B}(\{\alpha\})\}| \leq m_*$. This clearly suffices. 
	
	So let $\mathbf{c} \in \mathcal{B}$ be nonzero. Choose $\mathbf{c}_0 < \mathbf{c}$ nonzero, such that there is some $m_* < \omega$ with $\mathbf{c}_0 \leq \|\mathbf{m} = m_*\|_{(\omega, <)^{\mathcal{B}}}$. Try to find, by induction on $m \leq m_*+1$, a descending sequence $(\mathbf{c}_m: m \leq m_*+1)$ such that for each $m$, $|\{\alpha <\lambda: \mathbf{c}_m \leq \mathbf{B}(\{\alpha\})\}| \geq m$. There must be some stage $m < m_*+1$ at which we cannot continue, since for any $s \in [\lambda]^{m_*+1}$, $\mathbf{c}_0 \wedge \mathbf{B}(s) = 0$ (since $\mathbf{B}(s) \leq \|\mathbf{n} > m_*\|_{(\omega, <)^\mathcal{B}}$). So we get some $m < m_*+1$ such that if we set $s = \{\alpha < \lambda: \mathbf{c}_m \leq \mathbf{B}(\{\alpha\})\}$, then for all $\alpha \not \in s$, $\mathbf{c}_m \leq \lnot \mathbf{B}(\{\alpha\})$. So $\mathbf{c}_m$ is desired.
	
	(B) implies (A): suppose $s \in [\lambda]^n$. Then for any $\mathbf{c} \leq \mathbf{B}(s)$ nonzero such that $\mathbf{c}$ decides each $\mathbf{B}(\{\alpha\})$, we have that $\mathbf{c} \leq \|\mathbf{n} \geq n\|_{(\omega, <)^{\mathcal{B}}}$. Since the set of all such $\mathbf{c}$ is dense below $\mathbf{B}(s)$, we must have that $\mathbf{B}(s) \leq \|\mathbf{n} \geq n \|_{(\omega, <)^{\mathcal{B}}}$.
\end{proof}

We have the following theorem connecting all of these notions. It is a translation of Observation 9.9 of \cite{ConsUltrPOV} into our context. Parente \cite{Parente1} \cite{Parente2} has independently proven that if $\mathcal{U}$ is $\lambda$-OK and is $\aleph_1$-incomplete, then $\mathcal{U}$ is $\lambda$-regular.

\begin{theorem}\label{FlexToInterp}
	Suppose $\mathcal{U}$ is an ultrafilter on the complete Boolean algebra $\mathcal{B}$. Then $\lambda_{\mathcal{U}}(\Delta(FDP))$ is the least $\lambda$ such that $\mathcal{U}$ is not $\lambda$-OK. Additionally, if $\mathcal{U}$ is $\aleph_1$-incomplete, then this is the least $\lambda$ such that $\mathcal{U}$ is not $\lambda$-flexible.
\end{theorem}
\begin{proof}
	Choose some transitive $V \models ZFC^-$, and some $\mathbf{i}: V \preceq \mathbf{V}$ with $\mathbf{V}$ $\lambda^+$-saturated. Write $\hat{V} = \mathbf{V}/\mathcal{U}$ and let $\mathbf{j}: V \preceq \hat{V}$ be the usual embedding. 
	
	Suppose first that $\lambda < \lambda_{\mathcal{U}}(\Delta(FDP))$, and $\mathbf{A}$ is a $\lambda$-distribution with $\mathbf{A}(s) = \mathbf{A}(t)$ for all $|s| = |t|$. Then it is easy to see that $\mathbf{A}$ is a $(\lambda, \Delta(FDP))$-distribution, so by hypothesis $\mathbf{A}$ has a multiplicative refinement in $\mathcal{U}$; thus $\mathcal{U}$ is $\lambda$-OK. Conversely, suppose $\mathcal{U}$ is $\lambda$-OK; we show $\lambda < \lambda_{\hat{V}}(\Delta(FDP))$, using the characterization of Lemma \ref{nlowLemma1}. This suffices by Theorem~\ref{CardInvarTheoremUlt}. So suppose $\mathbf{m}_*,  \mathbf{n}_* \in \mathbf{i}(\omega)$ and $\{\mathbf{n}_\alpha: \alpha < \lambda\} \subseteq \mathbf{i}(\omega)$ are given with $[\mathbf{m}_*]_{\mathcal{U}}$ nonstandard and $[\mathbf{m}_*]_{\mathcal{U}} < [\mathbf{n}_*]_{\mathcal{U}}$, and each $[\mathbf{n}_\alpha]_{\mathcal{U}} < [\mathbf{n}_*]_{\mathcal{U}}$. We can suppose each $\|\mathbf{m}_* < \mathbf{n}_*\|_{\mathbf{V}}= \|\mathbf{n}_\alpha < \mathbf{n}_*\|_{\mathbf{V}} = 1$. Define a $\lambda$-distribution $\mathbf{A}$ in $\mathcal{U}$, via $\mathbf{A}(s) = \| \mathbf{m}_* \geq n\|_{\mathbf{V}}$ for each $s \in [\lambda]^n$. Since $\mathcal{U}$ is $\lambda$-OK, we can find a multiplicative refinement $\mathbf{B}$ of $\mathbf{A}$ in $\mathcal{U}$. By a similar argument to Lemma~\ref{CompactnessPatternsLemma2}, we can find $\mathbf{X} \in [\mathbf{n}_*]^{\leq \mathbf{m}_*}$ (i.e. $\mathbf{X} \in \mathbf{V}$ and $\|\mathbf{X} \in [\mathbf{n}_*]^{\leq \mathbf{m}_*}\|_{\mathbf{V}}= 1$) so that for all $\alpha < \lambda$, $\|\mathbf{n}_\alpha \in \mathbf{X}\|_{\mathbf{V}} = \mathbf{B}(\{\alpha\})$. Then $\hat{X} := [\mathbf{X}]_{\mathcal{U}}$ is as desired.

	Suppose next that $\mathcal{U}$ is $\lambda$-flexible; we show that $\mathcal{U}$ is $\lambda$-OK. So suppose $\mathbf{A}$ is a distribution in $\mathcal{U}$ such that for all $s, t \in [\lambda]^n$, $\mathbf{A}(s) = \mathbf{A}(t)$. If $\mathbf{a} := \bigwedge_n \mathbf{A}(n) \in \mathcal{U}$ then obviously the constant distribution with value $\mathbf{a}$ is a refinement in $\mathcal{U}$. Otherwise, we can suppose $\bigwedge_n \mathbf{A}(n) = 0$ (by intersecting each $\mathbf{A}(s)$ with $\lnot \mathbf{a}$). Define $\mathbf{m} \in (\omega,<)^{\mathcal{B}}$ via $\mathbf{m}(n) = \mathbf{A}(n) \wedge \lnot \mathbf{A}(n+1)$. The fact that $\mathbf{m} \in (\omega, <)^{\mathcal{B}}$ follows form  $\bigwedge_{n < \omega} \mathbf{A}(n) = 0$ and $\mathbf{A}(0) = \mathbf{A}(\emptyset) = 1$. $\mathbf{m}$ is $\mathcal{U}$-nonstandard since each $\|\mathbf{m} \geq n\|_{(\omega,<)^{\mathcal{B}}} = \mathbf{A}(n) \in \mathcal{U}$. Thus we can find a multiplicative distribution $\mathbf{B}$ in $\mathcal{U}$, such that for all $s \in [\lambda]^n$, $\mathbf{B}(s) \leq \mathbf{A}(s) = \|\mathbf{m} \geq n\|_{\mathbf{V}}$.
	
	Finally, suppose $\mathcal{U}$ is $\lambda$-OK and $\aleph_1$-incomplete, and let $[\mathbf{m}]_{\mathcal{U}}$ be a $\mathcal{U}$-nonstandard element of $(\omega, <)^{\mathcal{B}}/\mathcal{U}$. Define $\mathbf{A}(s) = \| \mathbf{m} \geq n\|_{(\omega, <)^{\mathcal{B}}}$ for each $s \in [\lambda]^{n}$ and let $\mathbf{B}$ be a multiplicative refinement of $\mathbf{A}$ in $\mathcal{U}$. Then for all $s \in [\lambda]^n$, $\mathbf{B}(s) \leq \|\mathbf{m} \geq n\|_{\mathbf{V}} = \mathbf{A}(s)$. 
\end{proof}

\begin{remark}\label{OKStableSat}
	It follows that if $\mathcal{U}$ is $\lambda$-OK, then $\mu_{\mathcal{U}} > \lambda$, and so $\mathcal{U}$ $\lambda^+$-saturates every stable theory. The special case where $\mathcal{B} = \mathcal{P}(\lambda)$ and $\mathcal{U}$ is $\aleph_1$-incomplete (i.e. $\lambda$-flexible) was proved similarly by Malliaris and Shelah in \cite{IndFunctUlts}. Hence: if $\mathcal{U}$ is $\aleph_1$-complete (and thus $\lambda$-OK for all $\lambda$) and if $T$ is stable, then $\lambda_{\mathcal{U}}(T) = \infty$, i.e. $\mathcal{U}$ $\lambda^+$-saturates $T$ for every $\lambda$.
\end{remark}

\begin{corollary}\label{nonlowUlt}
	Suppose $\mathcal{U}$ is an ultrafilter on $\mathcal{B}$ and $\lambda$ is given. Then the following are equivalent:
	
	\begin{itemize}
		\item[(A)] $\mathcal{U}$ $\lambda^+$-saturates $T_{nlow}$;
		\item[(B)] $\mathcal{U}$ $\lambda^+$-saturates some nonlow theory;
		\item[(C)] $\mathcal{U}$ $\lambda^+$-saturates $T_{rg}$ and $\mathcal{U}$ is $\lambda$-OK.
	\end{itemize}
\end{corollary}

\section{The Chain Condition and Saturation}\label{SurveyCCSec}

Malliaris and Shelah prove some special cases of the following in \cite{Optimals}, but their arguments do not generalize.

\begin{theorem}~\label{SimpleNonSat}
	Suppose $\mathcal{B}$ is a complete Boolean algebra; write $\lambda = \mbox{c.c.}(\mathcal{B})$. Suppose $\mathcal{U}$ is a nonprincipal ultrafilter on $\mathcal{B}$. Then $\mathcal{U}$ does not $\lambda^+$-saturate any nonsimple theory.  In fact, we can find a $(\lambda, \Delta(TP_2))$-distribution $\mathbf{A}$ in $\mathcal{U}$, such that if $\mathcal{B}$ is a complete subalgebra of $\mathcal{B}_*$ where $\mathcal{B}_*$ has the $\lambda$-c.c., then $\mathbf{A}$ has no multiplicative refinement in $\mathcal{B}_*$.
\end{theorem}
\begin{proof}
	
	It suffices to show the second claim. Note that $\lambda > \aleph_0$, as otherwise $\mathcal{B}$ would be finite, and so would not admit any nonprincipal ultrafilters. Thus $\lambda$ is regular; this is a theorem due to Erd\"{o}s and Tarski \cite{ErdosTarski}, or see Theorem 7.15 of Jech \cite{Jech}.
	
	Let $\sigma$ be the completeness of $\mathcal{U}$, i.e. the least cardinal such that there is a descending sequence $(\mathbf{a}_\alpha: \alpha < \sigma)$ from $\mathcal{U}$ with $\bigwedge_{\alpha < \sigma} \mathbf{a}_\alpha = 0$. It is not hard to see that $\sigma < \lambda$, and moreover there is an antichain $\mathbf{C}$ of $\mathcal{B}$ of size $\sigma$ such that for every $X \in [\mathbf{C}]^{<\sigma}$, $\bigvee X \not \in \mathcal{U}$. Enumerate $\mathbf{C} = (\mathbf{c}_\gamma: \gamma < \sigma)$.
	
	Let $S \subseteq \lambda$ be the set of all $\alpha < \lambda$ with $\mbox{cof}(\alpha) = \sigma$, so $S$ is stationary in $\lambda$. For each $\alpha \in S$, let $L_\alpha: \sigma \to \alpha$ be a cofinal, increasing map, and let $\underline{\delta}_\alpha \in (\lambda, <)^{\mathcal{B}}$ be the element such that for all $\gamma < \sigma$, $\|\underline{\delta}_\alpha = L_\alpha(\gamma)\|_{(\lambda,<)^{\mathcal{B}}} = \mathbf{c}_\gamma$. This determines $\underline{\delta}_\alpha$, since $\mathbf{C}$ is a maximal antichain. In particular, we have that $\|\underline{\delta}_\alpha < \alpha\|_{(\lambda, <)^{\mathcal{B}}} = 1$, and for all $\beta < \alpha$, $\|\underline{\delta}_\alpha > \beta\|_{(\lambda,<)^{\mathcal{B}}} \in \mathcal{U}$. In particular, for all $\alpha < \beta$ both in $S$, $\|\underline{\delta}_\alpha < \underline{\delta}_\beta\|_{(\lambda,<)^{\mathcal{B}}} \in \mathcal{U}$.

	For each $s \in [\lambda]^{<\aleph_0}$, put $\mathbf{A}(s) = \bigwedge_{\alpha \not= \beta \in s} \| \underline{\delta}_\alpha \not= \underline{\delta}_\beta \|_{(\lambda,<)^{\mathcal{B}}}$; so $\mathbf{A}$ is a $(\lambda, \Delta(TP_2))$-distribution in $\mathcal{U}$. Suppose  $\mathcal{B}$ is a complete subalgebra of $\mathcal{B}_*$ where $\mathcal{B}_*$ has the $\lambda$-c.c. We show that $\mathbf{A}$ has no multiplicative refinement in $\mathcal{B}_*$, i.e. there is no multiplicative $\lambda$-distribution $\mathbf{B}$ in $\mathcal{B}_*$ such that for all $\alpha <\beta$, $\mathbf{B}(\{\alpha, \beta\}) \leq \|\underline{\delta}_\alpha \not= \underline{\delta}_\beta \|_{(\lambda,<)^{\mathcal{B}}}$.
	
	Suppose there were. For each $\alpha < \lambda$ there is some $f(\alpha) < \sigma$ with $\mathbf{B}(\{\alpha\}) \wedge \mathbf{c}_{f(\alpha)}$ nonzero, i.e. with $\mathbf{B}(\{\alpha\}) \wedge \|\underline{\delta}_\alpha = L_\alpha(f(\alpha))\|_{(\lambda, <)^{\mathcal{B}}/\mathcal{U}}$ nonzero. Write $g(\alpha) = L_\alpha(f(\alpha)) < \alpha$. By Fodor's Lemma (using that $\lambda$ is regular), we can find a stationary set $S' \subseteq S$ on which $g$ is constant, say with value $\gamma$. Since $\mathcal{B}_*$ has the $\lambda$-c.c., $(\mathbf{B}(\{\alpha\}) \wedge \|\underline{\delta}_\alpha = \gamma\|_{(\lambda, <)^{\mathcal{B}}/\mathcal{U}}: \alpha \in S')$ is not an antichain, so we can choose $\alpha < \beta$ both in $S'$ such that $\mathbf{B}(\{\alpha\}) \wedge \mathbf{B}(\{\beta\}) \wedge \|\underline{\delta}_\alpha = \gamma\|_{(\lambda, <)^{\mathcal{B}}/\mathcal{U}} \wedge \|\underline{\delta}_\beta = \gamma\|_{(\lambda, <)^{\mathcal{B}}/\mathcal{U}}$ is nonzero. But $\mathbf{B}(\{\alpha, \beta\}) \leq \|\underline{\delta}_\alpha \not= \underline{\delta}_\beta\|_{(\lambda, <)^{\mathcal{B}}/\mathcal{U}}$, a contradiction.
\end{proof}

I proved the following (minus the ``in fact" clause) in \cite{LowDividingLine}; that result in turn built off of some special cases proven by Malliaris and Shelah in \cite{DividingLine}.
\begin{theorem}\label{LowNonSat}
	Suppose $\mathcal{B}$ is a complete Boolean algebra; write $\lambda = \mbox{c.c.}(\mathcal{B})$. Suppose $\mathcal{U}$ is an $\aleph_1$-incomplete ultrafilter on $\mathcal{B}$. Then $\mathcal{U}$ does not $\lambda^+$-saturate any nonlow theory. In fact, there is a $(\lambda, \Delta(FDP))$-distribution $\mathbf{A}$ in $\mathcal{U}$ such that if $\mathcal{B}$ is a complete subalgebra of $\mathcal{B}_*$ and $\mathcal{B}_*$ has the $\lambda$-c.c., then $\mathbf{A}$ has no multiplicative refinement in $\mathcal{B}_*$.
\end{theorem}
\begin{proof}
	It suffices to show the second claim. Note that $\lambda > \aleph_0$, as otherwise $\mathcal{B}$ would be finite, and so would not admit any nonprincipal ultrafilters. 
	
	Let $\mathcal{U}$ be an $\aleph_1$-incomplete ultrafilter on $\mathcal{B}$; then we can choose a descending sequence $(\mathbf{c}_n: n < \omega)$ from $\mathcal{U}$ such that $\mathbf{c}_0 = 1$ and $\bigwedge_{n} \mathbf{c}_n = 0$. Let $\mathbf{A}$ be the distribution in $\mathcal{U}$, defined by $\mathbf{A}(s) = \mathbf{c}_{|s|}$. Then $\mathbf{A}$ is a $(\lambda, \Delta(FDP))$-distribution. Suppose $\mathcal{B}_*$ has the $\lambda$-c.c. and $\mathcal{B}$ is a complete subalgebra of $\mathcal{B}_*$. If $\mathbf{A}$ has a multiplicative refinement $\mathbf{B}$ in $\mathcal{U}$, then by Theorem~\ref{FlexEquiv}, $\mathbf{B}$ would be a $\lambda$-regular distribution. But we note in \cite{BVModelsUlrich} that no complete Boolean algebra with the $\lambda$-c.c. admits a $\lambda$-regular family. 
\end{proof}

It will be convenient for us in \cite{AmalgKeislerUlrich} if we phrase our results in terms of $(\lambda, T)$-{\L}o{\'s} maps instead of $(\lambda, \Delta)$-distributions. We recall from Remark~\ref{LosRemark} that if $T$ admits $\Delta$ and $\mathbf{A}$ is a $(\lambda, \Delta)$-distribution, then $\mathbf{A}$ is a $(\lambda ,T)$-{\L}o{\'s} map.

\begin{corollary}\label{nonsatLos}
	In Theorem~\ref{SimpleNonSat} and ~\ref{LowNonSat}, the distribution $\mathbf{A}$ is a $(\lambda, T_{rf})$-{\L}o{\'s} map or a $(\lambda, T_{nlow})$-{\L}o{\'s} map, respectively.
\end{corollary}
\begin{proof}
	Since $T_{rf}$ admits $\Delta(TP_2)$, every $(\lambda, \Delta(TP_2))$-distribution is a $(\lambda, T_{rf})$-{\L}o{\'s} map; and since $T_{nlow}$ admits $\Delta(FDP)$, every $(\lambda, \Delta(FDP))$-distribution is a $(\lambda, T_{nlow})$-{\L}o{\'s} map.
\end{proof}

\begin{remark}
	
	It follows from results in \cite{InterpOrdersUlrich} that if $T$ is a complete countable theory with the NFCP, then any ultrafilter $\mathcal{U}$ satisfies $\lambda_{\mathcal{U}}(T) = \infty$, but if $T$ has FCP and $\mathcal{U}$ is $\aleph_1$-incomplete, then $\lambda_{\mathcal{U}}(T) \leq |\mathcal{B}|$ (see Remark 9.9 and Corollary 9.12). Remark~\ref{OKStableSat} of the present work says that if $\mathcal{U}$ is $\aleph_1$-complete and if $T$ is stable, then $\lambda_{\mathcal{U}}(T) = \infty$, and Theorem~\ref{ExistenceConverseFirst} says that if $\mathcal{U}$ is nonprincipal and $T$ is unsimple, then $\lambda_{\mathcal{U}}(T) \leq \mbox{c.c.}(\mathcal{B})$. Malliaris and Shelah show in \cite{IndFunctUlts} that if $\mathcal{U}$ is an ultrafilter on $\mathcal{P}(\lambda)$ with $|\lambda^\lambda/\mathcal{U}| = 2^\lambda$ and if $T$ is unstable, then $\lambda_{\mathcal{U}}(T) \leq 2^\lambda$; in particular, this holds whenever $\lambda$ is inaccessible and $\mathcal{U}$ is a uniform, $\aleph_1$-complete ultrafilter on $\mathcal{P}(\lambda)$.
	
	But the following is open:
\end{remark}

\noindent \textbf{Conjecture.} Suppose $\mathcal{B}$ is a complete Boolean algebra and $\mathcal{U}$ is a nonprincipal ultrafilter on $\mathcal{B}$, and suppose $T$ is unstable. Then $\lambda_{\mathcal{U}}(T) \leq |\mathcal{B}|$.
%

\bibliography{mybib}

\end{document}